\documentclass[11pt]{article}
\usepackage{a4}
\usepackage{amsmath,amssymb,amsthm}
\usepackage{amsfonts}
\usepackage{url,color,graphicx}

\definecolor{blue}{rgb}{0,0,0.7}
\definecolor{M}{rgb}{0.8,0,0.8}

\newtheorem{theorem}{Theorem}[section]
\newtheorem{lemma}[theorem]{Lemma}

\newtheorem{corollary}[theorem]{Corollary}
\newtheorem{example}[theorem]{Example}
\newtheorem{remark}[theorem]{Remark}
\newtheorem{definition}[theorem]{Definition}

\numberwithin{equation}{section}

\newcommand{\LL}{{\mathcal L}}
\newcommand{\pt}{\partial}
\newcommand{\bR}{\mathbb{R}}
\newcommand{\RR}{{\mathcal R}}
\newcommand{\EE}{{\mathcal E}}
\newcommand{\FF}{{\mathcal F}}

\newcommand {\beq} {\begin{equation}}
\newcommand {\eeq} {\end{equation}}

\newcommand{\TOL}{{\it TOL}}
\newcommand{\ind}{\hspace*{0.5cm}}

\newcommand{\tcb}[1]{\textcolor{blue}{#1}}

\title{A posteriori error analysis for variable-coefficient multiterm time-fractional  subdiffusion  equations
\thanks{The research of Natalia Kopteva is supported in part by Science Foundation Ireland under grant 18/CRT/6049.
The research of Martin Stynes is supported in part by the National Natural Science Foundation of China under grants 12171025 and NSAF-U1930402.}}

\author{Natalia Kopteva\thanks{Department of Mathematics and Statistics, University of Limerick, Limerick V94 T9PX, Ireland  (\texttt{natalia.kopteva@ul.ie}).}
\and
Martin Stynes\thanks{Applied and Computational Mathematics Division, Beijing Computational Science Research Center,
Beijing 100193, China (\texttt{m.stynes@csrc.ac.cn}). Corresponding author.}
}

\begin{document}

\maketitle

\begin{abstract}
An initial-boundary value problem of subdiffusion type is considered;  the temporal component of the differential operator has the form $\sum_{i=1}^{\ell}q_i(t)\, D _t ^{\alpha_i} u(x,t)$, where the $q_i$ are continuous functions, each $D _t ^{\alpha_i}$ is a Caputo derivative, and the $\alpha_i$ lie in $(0,1]$. Maximum/comparison principles for this problem are proved under weak hypotheses.  A new positivity result for the multinomial Mittag-Leffler function is derived. A posteriori error bounds are obtained in $L_2(\Omega)$ and $L_\infty(\Omega)$, where the spatial domain $\Omega$ lies in $\bR^d$ with $d\in\{1,2,3\}$. An adaptive algorithm based on this theory is tested extensively and shown to yield accurate numerical solutions on the meshes generated by the algorithm.
\end{abstract}

\noindent\emph{Keywords:} multiterm time-fractional, variable coefficient, subdiffusion, a posteriori error analysis \\
\noindent\emph{AMS MSC Classification:}65M15

\maketitle

\section{Introduction}
The numerical solution of \tcb{fractional} differential equations (FDEs) is currently the subject of much research (see for example \cite{JLZ19,StyL1survey}), since such equations model many physical processes but their exact solution is generally impossible. Of course this is also true for classical integer-order differential equations, where mesh-adaptive numerical methods based on a posteriori error analyses have played a significant role for many years. Methods of this type have very general usefulness since they require no knowledge of the properties of the unknown solution to the problem. But for FDEs, there has been little progress in theory-based adaptive numerical methods; their development has been impeded by the absence of a satisfactory a posteriori theory for their error analysis.

As it is often difficult to analyse the regularity and other fundamental properties of the unknown solutions to FDEs, it can be impossible to give any rigorous a priori analysis error analysis of numerical methods for their solution. This makes it even more desirable to devise an a posteriori error analysis that does not require any information about the unknown solution.

Recently a new and very promising a posteriori error estimation methodology for FDEs appeared in \cite{KopAML22}, which considered initial-value and initial-boundary value time-fractional  subdiffusion  problems whose differential equations contained a single temporal derivative of fractional order.  It is clearly desirable to extend this theory to time-fractional FDEs containing several fractional derivatives, as these offer more powerful modelling capabilities. Our primary aim in the current paper is to develop the a posteriori theory for this extension and to show experimentally that an adaptive algorithm based on our theory is able to compute accurate numerical solutions to problems whose solutions have singularities (as is usually the case with FDEs). It should be noted that these accurate solutions are computed on nonuniform meshes that are constructed automatically by the algorithm --- the user does not have to provide any special mesh, nor input any attributes of the unknown solution.

\tcb{The relationship between our paper, which studies a multiterm fractional derivative operator, and  \cite{KopAML22}, where only a single fractional derivative appears, is the following.
 Section~\ref{sec:L2apost} below points out similarities between Lemma~\ref{lem_aux}, Theorem~\ref{the_L2} and Corollary~\ref{cor:L2boundEE0} and results from~[15]; but  while Corollary~\ref{cor:L2boundEE1} is analogous to the second bound in \cite[Corollary 2.4]{KopAML22}, the proof of  Corollary~\ref{cor:L2boundEE1} is much deeper since it involves hypergeometric functions whereas  \cite{KopAML22} needed only elementary functions.
Outside Section~\ref{sec:L2apost} there are significant differences between our paper and~\cite{KopAML22} --- see Theorem~\ref{thm:wlowerbound}, Remark~\ref{rem:w'(t)}, Lemma~\ref{lem:beta<1}, eq.~\eqref{w1}; Lemma~\ref{lem:special} would be trivial in the single-term case; the multinomial Mittag-Leffler function of Definition~\ref{def:multiML} that is needed for the multiterm case is less tractable than the more familiar two-parameter Mittag-Leffler function that suffices for the single-term case   --- thus all of the rather technical Appendix~\ref{sec:app} is new. }

The paper is structured as follows. Section~\ref{sec:multitermprob} describes the multiterm time-fractional initial-boundary value problem of  subdiffusion  type that will be studied. In Section~\ref{sec:nonnegative}, maximum/comparison principles and some of their consequences are derived for the associated fractional initial-value problem; existence of a solution for that problem is also discussed. A posteriori error bounds for $L_2(\Omega)$, where the spatial domain $\Omega$ lies in $\bR^d$ with $d\in\{1,2,3\}$, are established in Section~\ref{sec:L2apost}.  A variant of this theory in Section~\ref{sec:Linf} gives a posteriori error bounds in  $L_\infty(\Omega)$. Then in Section~\ref{sec:L1method} we perform extensive numerical experiments to demonstrate the effectiveness and reliability of the theory of Sections~\ref{sec:L2apost} and~\ref{sec:Linf}. Finally, an Appendix proves a new positivity result for the multinomial Mittag-Leffler function, then uses it to give an alternative version of a result from Section~\ref{sec:nonnegative}.

\subsection{The multiterm time-fractional  subdiffusion  problem}\label{sec:multitermprob}
We shall study the multiterm time-fractional  subdiffusion  problem
\begin{subequations} \label{problem}
\begin{align}
\sum_{i=1}^{\ell}\bigl[q_i(t)\, D _t ^{\alpha_i} u(x,t)\bigr]+ \LL u(x,t) &=f(x,t)\quad \text{for} \;\; (x,t)\in \Omega\times(0,T], \label{pde} \\
\intertext{ with initial and boundary conditions}
u(x,0)=u_0(x)\quad \text{for} \;\;x\in \Omega, \quad
u(x,t)&=0\quad\text{for} \;\;x\in \partial\Omega \;\;\text{and} \;\;0<t\leq T.  \label{bc}
\end{align}
\end{subequations}
Here
$\ell$ be a positive integer, the constants $\alpha_i$ (for $i=1,2,\dots, \ell$) satisfy
\begin{equation}\label{alphai}
0<\alpha_\ell<...<\alpha_2< \alpha_1 \le 1,
\end{equation}
while each $q_i\in C[0,T]$ with
\beq
\sum_{i=1}^{\ell}q_i(t)>0\quad \text{and} \quad q_i(t)\ge 0,\;\; i=1,\ldots,\ell,  \label{q_i}
\quad \text{for} \;\; t\in[0,T].
\eeq
This problem is posed in a bounded Lipschitz domain  $\Omega\subset\bR^d$ (where $d\in\{1,2,3\}$), and involves
a spatial linear second-order elliptic operator~$\LL$.
Each Caputo temporal fractional derivative $D_t^{\alpha_i}$ is  defined \cite{Diet10} for $0<\alpha_i<1$ and $t>0$ by
\begin{equation}\label{CaputoEquiv}
D_t^{\alpha_i} u := J_t^{1-\alpha_i}(\pt_t u),\qquad
J_t^{1-\alpha_i} y(\cdot,t) :=  \frac1{\Gamma(1-\alpha_i)} \int_{0}^t(t-s)^{-\alpha_i}\, y(\cdot, s)\, ds,
\end{equation}
where $\Gamma(\cdot)$ is the Gamma function, and $\pt_s$ denotes the partial derivative in~$s$.
From \cite[Theorem 2.20 and Lemma 3.4]{Diet10} it follows that $\lim_{\alpha\rightarrow 1^-}D_t^{\alpha} u(x,t)=\pt_t u(x,t)$ for each $(x,t)\in\Omega\times (0,T]$ when $u(x,\cdot)\in C^1[0,T]$, so for $\alpha_1=1$ we take $D_t^{\alpha_1} u=D_t^{1} u:=\pt_t u$.

\tcb{
\begin{remark}
One might wonder whether the presence of lower-order fractional derivatives in the differential operator would invalidate the above presumption that $u(x,\cdot)\in C^1[0,T]$ if  $\alpha_1=1$, but when $\alpha_1=1$ (and $q_1(t)>0$ for all $t$) and the data  are continuous, in Lemma~\ref{lem:special} we prove that the solution of the associated initial-value problem does lie in $C^1[0,T]$. See also Remark~\ref{rem:w'(t)}, where it is shown that  if $\alpha_1=1$ then at $t=0$ the solution is better behaved than if $\alpha_1<1$. Furthermore, in the case of constant coefficients $q_i$, when $\alpha_1=1$ one can deduce that the solution of the initial-value problem lies in $C^1[0,T]$ from the explicit solution given by Remark~\ref{rem:IVPconstsoln} and eq.~\eqref{w1}, though we omit the details.
\end{remark}
}

In the case where each $q_i$ is a positive constant and $\alpha_1<1$, existence of a solution to \eqref{problem} follows from \cite[Theorems 2.1 and 2.2]{LLY15}.
For the general case of variable $q_i$ satisfying \eqref{q_i}, one can show uniqueness of a solution to \eqref{problem} by imitating the argument of \cite[Theorem 4]{Luch11}.

The problem \eqref{problem} with constant $q_i$ was considered in \cite{ChenSty2022,LLY15} and their references.
\tcb{Two-term fractional differential equations (i.e., $\ell=2$ in~\eqref{q_i}) appear in \cite{MKS98} modelling anomalous transport and in \cite{SBMB03} modelling solute transport in aquifers. In~\cite[eq.~(14)]{SBMB03}, the time-fractional PDE
\begin{equation}\label{aquifer}
\partial_t C + \beta D_t^\alpha C - \LL C =0
\end{equation}
is used to model solute transport in aquifers, where $C= C(x,t)$ denotes concentration and $\alpha\in (0,1)$. This is the particular case of our fractional PDE~\eqref{pde} where $\ell=2$, $\alpha_1=1$ and $ \alpha_2=\alpha$, with $q_1=1$ and $q_2=\beta>0$ so~\eqref{q_i} is satisfied. The ``fractal immobile capacity" $\beta$ in~\eqref{aquifer} may be time-dependent; for example in~\cite[Figure~4]{SBMB03} the authors take $\beta = 0.08 d^{-0.67}$ where $d$ is time measured in days. Thus it is of interest to consider time-dependent $q_i$ in~\eqref{problem}.  }

\tcb{
Alternatively,  to incorporate uncertainties in the data  of the physical problem, one can use a variably distributed-order subdiffusion problem like that of~\cite{YZW20}, where the distributed-order fractional derivative is defined by $\tilde D_t^\rho u(x,t) = \int_0^1 \rho(\alpha) D_t^\alpha u(x,t)\,d\alpha$ with $\rho=\rho(\alpha)$  a probability density function. To handle this numerically one must apply a quadrature rule to $\tilde D_t^\rho u$, which can lead to a PDE such as~\eqref{pde} that satisfies the hypothesis~\eqref{q_i}.
}

It appears that  \eqref{problem} with variable $q_i(t)$  has never been \tcb{rigorously} studied in the \tcb{mathematics} literature. This variant, however,  is of some interest since it is a simple (and hence attractive) alternative to models with variable-fractional-order equations, which
have received a lot of attention in recent years (see \cite{ZW_IMAJNA2021} and its references).

\smallskip

\noindent{\it Notation.} We use the standard inner product $\langle\cdot,\cdot\rangle$ and the norm $\|\cdot\|$
in the space $L_2(\Omega)$, as well as the standard spaces 
$L_\infty(\Omega)$, $H^1_0(\Omega)$,
$L_{\infty}(0,t;\,L_2(\Omega))$, and $W^{1,\infty}(t',t'';\,L_2(\Omega))$ (see \cite[Section 5.9.2]{Evans10} for the notation used for functions of $x$ and~$t$).
The notation $v^+:=\max\{0,\,v\}$ is used for the positive part of a generic function~$v$. For convenience we sometimes write
\beq\label{DD_t_def}
D_t^{\bar\alpha} :=  \sum_{i=1}^{\ell}q_i(t) D _t ^{\alpha_i}.
\eeq

\section{Nonnegative solutions of certain initial-value problems}\label{sec:nonnegative}
Our a posteriori analysis will rely on the property that the solutions of certain multiterm fractional  initial-value problems are nonnegative; we derive this result in this section after presenting a reformulation of the definition~\eqref{CaputoEquiv} of the fractional derivative $D_t^{\alpha_i}y(\cdot, t)$ that can be applied to a more general class of functions.

For simplicity, in this section we write $y(t)$ instead of $y(x, t)$ since  the dependence on $x$ is irrelevant here.

\subsection{Function regularity and reformulated Caputo derivative}\label{sec:fnregular}
In~\eqref{CaputoEquiv} one can integrate by parts  to reformulate the definition of $D_t^{\alpha_i}y(t)$ for $\alpha_i<1$ as
\begin{equation}\label{altCaputo}
\Gamma(\tcb{1-\alpha_i})\,D_t^{\alpha_i} y(t)
=t^{-\alpha_i} \left[y(t)-y(0)\right]  + \int_{0}^t\!\alpha_i(t-s)^{-\alpha_i-1}\, \left[y(t)- y(s)\right]\, ds
\end{equation}
for $0<t \le T$.
This reformulation appeared already in~\cite[eq.~(2.4)]{KopAML22}, and in, e.g., \cite[Lemma 3.1]{BHY15}, \cite[Lemma 2.10]{JinBook},  \cite[Proof of Theorem 1]{Luch09}, and \cite[Theorem 5.2]{Vai16}. We will show that it has the advantage that it permits the use of less smooth functions~$y$ than \eqref{CaputoEquiv}; this attribute is needed, for example, to prove Lemma~\ref{lem_aux} below.

Recall that~\eqref{altCaputo} was obtained from~\eqref{CaputoEquiv} by integration by parts. From the proof of the integration by parts formula, one sees that this calculation  is valid if for each $t'\in [0,t)$  the function $\psi(t;\cdot)$ defined by $\psi(t;s) := (t-s)^{-\alpha_i}[y(s)-y(t)]$ is absolutely continuous on $[0,t']$ and satisfies $\lim_{t'\to t^-}\psi(t;t')=0$, because one can  integrate by parts $\int_{0}^{t'}(t-s)^{-\alpha_i}\, y'(s)\, ds$, then take $\lim_{t'\to t^-}$.

For example, if $y$ lies in the standard H\"older space $C^\beta[0,T]$ for some $\beta>\alpha_i$, then this derivation of~\eqref{altCaputo}  from~\eqref{CaputoEquiv} is valid; see \cite[Lemma 3.1]{BHY15}.

As in  \cite{KopAML22}, we consider now a more general class of functions for which the definition~\eqref{CaputoEquiv} is unsuitable but~\eqref{altCaputo} can be used.

Let us assume that
\begin{equation}\label{ycondition}
y\in C[0,T]\cap W^{1,\infty}(\epsilon,t)\ \text{ for all }\epsilon, t\ \text{ satisfying }0<\epsilon<t\le T.
\end{equation}
The hypothesis that $y\in W^{1,\infty}(\epsilon,t)$ is equivalent to assuming that $y$ is Lipschitz continuous on each interval $[\epsilon, t]$; see~\cite[p.154]{GTbook}.
\tcb{If $\alpha_1=1$, then we strengthen~\eqref{ycondition} by assuming that $y\in C[0,T]$ and $y'$ is a left-continuous function on~$(0,T]$ that may have jump discontinuities; see Section~\ref{sec:IVP}.}

Fix $t\in (0,T]$. The integral $ \int_0^{t/2}\!\alpha_i(t-s)^{-\alpha_i-1}\, \left[y(t)- y(s)\right]\, ds$ is defined and finite as its integrand lies in $C\left[0, \frac12 t\right]$. For  $ \int_{t/2}^t\!\alpha_i(t-s)^{-\alpha_i-1}\, \left[y(t)- y(s)\right]\, ds$, since $y\in W^{1,\infty}\left(\frac12t,t\right)$ one has $\vert y(t)- y(s) \vert \le C(t-s)$ with a constant $C$ that depends on $t$ but is independent of $s$, which implies that the integral exists and is finite. Thus for all $y$ satisfying~\eqref{ycondition}, we can define $D_t^{\alpha_i} y(t)$   by~\eqref{altCaputo}.

The next two remarks describe weakenings of the hypothesis \eqref{ycondition}  on the function~$y$ that still  allow us to define $D_t^{\alpha_i} y(t)$  by~\eqref{altCaputo}.

\begin{remark}\label{rem:weakerHolder}
One could replace $y\in W^{1,\infty}(\epsilon,t)$ in \eqref{ycondition} by $y\in C^{\beta}(\epsilon, T]$ for all $\epsilon, t$  satisfying $0<\epsilon<t\le T$, where $\beta$ is any constant satisfying $\alpha_i < \beta \le 1$ and  $C^{\beta}(\epsilon, T]$ is a standard H\"older space.
\end{remark}

\begin{remark}[initial discontinuity in $y$]\label{rem:limt0}
Note that $y\in W^{1,\infty}(\epsilon,t)$ for all $\epsilon, t$ satisfying $0<\epsilon<t\le T$ implies $y\in C(0,T]$.
In \eqref{ycondition} one could replace the hypothesis that $y\in C[0,T]$ by an assumption that  $y\in L_\infty(0,T)$, and still work with~\eqref{altCaputo}. In particular we can replace $C[0,T]$ in (2.2) by an assumption that $\lim_{t\to 0^+}y(t)$ exists;  this will be useful in the forthcoming error analysis.
\end{remark}

 \subsection{The initial-value problem}\label{sec:IVP}
 Consider the initial-value problem
\beq\label{IVP}
D_t^{\bar\alpha}  w(t) + \lambda w(t) = v(t)\ \text{ for }0<t\le T,  \quad w(0)=w_0,
\eeq
where we assume that the parameter $\lambda\ge 0$. (We shall use the notation $\tcb{u(x,t)}$ for the solution of~\eqref{problem} and $w(t)$ for the solution of~\eqref{IVP}.)

In the next lemma we specify hypotheses allowing, for example, the possibility that $w$ is a piecewise polynomial.
Given a function $g$ that has a jump discontinuity at a finite number of points in $(0,T)$ but is continuous otherwise on $(0,T]$, at any point of discontinuity $\tau$ we take $g(\tau)= \lim_{t\to \tau^-}g(t)$. That is, we regard $g$ as left-continuous on $(0,T]$.

\begin{lemma}[Comparison principle for the initial-value problem]\label{lem:v0w0}
Consider \eqref{IVP} where $v\ge 0$ may have a finite number of jump discontinuities in $(0,T)$ and is left-continuous on $(0,T]$.
Suppose that $w$~satisfies the regularity hypothesis \eqref{ycondition}. Define $D_t^{\alpha_i} w$ by~\eqref{altCaputo} if $\alpha_i<1$. If $\alpha_1=1$,  suppose also that $w'$ may have jump discontinuities but is a left-continuous function on $(0,T]$. Assume that $w_0\ge 0$.
Then $w(t)\ge 0$ for $t\in[0,T]$.
\end{lemma}
\begin{proof}
Suppose that the result is false. Then since $w\in C[0,T]$ and $w(0)\ge 0$, there exists a point $t_0\in (0,T]$ such that $w(t_0)<0 \le w(0)$ and $w(t_0)\le w(t)$ for all~$t\in [0,T]$. From~\eqref{altCaputo} one sees immediately that each  $D_t^{\alpha_i} w(t_0) < 0$  if $\alpha_i<1$, while if $\alpha_i=1$ then $w'(t_0)\le 0$ (consider the interval $[0, t_0]$ and use the left-continuous property of $w'(t)$ at $t_0$). Hence $D_t^{\bar\alpha}  w(t_0) + \lambda w(t_0)<0 \le v(t_0)$, so $w$ cannot be a solution of~\eqref{IVP}.
(The case where $q_i(t_0)=0$ for $i=2,3,\dots,\ell$ is exceptional, as we then get only $w'(t_0) + \lambda w(t_0) \le 0$; to derive a contradiction, one can make a change of variable $\tilde w(t) := e^{-\mu t}w(t)$ for suitable~$\mu$ as in \cite[Section 2]{KopMaxPrin} and consider the initial-value problem satisfied by~$\tilde w$.)
\end{proof}

The following extension of Lemma~\ref{lem:v0w0} weakens the requirement that $w\in C[0,T]$. It will be needed to deal with the discontinuous function~$\EE_0$ of  Section~\ref{sec:L2apost}.

\begin{corollary}\label{cor:wnotcont}
In Lemma~\ref{lem:v0w0}, replace the hypothesis that $w\in C[0,T]$ by $\lim_{t\to 0^+}w(t)\ge 0$ exists. Then $w(t)\ge 0$ for $t\in(0,T]$.
\end{corollary}
\begin{proof}
Recalling Remark~\ref{rem:limt0}, one can use the same argument as for Lemma~\ref{lem:v0w0}, with minor modifications.
\end{proof}

We now use Lemma~\ref{lem:v0w0} to derive a stronger bound on $w$.
First, recall the well-known two-parameter Mittag-Leffler function
$E_{\alpha,\beta}(s) := \sum_{k=0}^\infty s^k/\Gamma(\alpha k+\beta)$.

\begin{theorem}\label{thm:wlowerbound}
Consider \eqref{IVP}, where $v\in C[0,T]$ with $v\ge 0$ and $w_0\ge 0$.
Suppose that $w$~satisfies the regularity hypothesis \eqref{ycondition}. Define $D_t^{\alpha_i} w$ by~\eqref{altCaputo} if $\alpha_i<1$. If $\alpha_1=1$,  suppose also that $w\in C^1(0,T]$.
Set $\underline q_j = \min_{t\in [0,T]}q_j(t)$ for $j=1,\dots,\ell$.
Then
\begin{equation}\label{wtEa}
w(t)\ge w_0 \max_{j=1,\dots, \ell}E_{\alpha_j,1}(-\lambda t^{\alpha_j}/\underline q_j)\quad\text{for }t\in[0,T],
\end{equation}
where one sets $E_{\alpha_j,1}(-\lambda t^{\alpha_j}/\underline q_j) \equiv 0$ if $\underline q_j=0$.
\end{theorem}
\begin{proof}
Fix $j\in \{1,\dots,\ell\}$. Define the barrier function~$B_j$ by $\underline q_jD _t ^{\alpha_j} B_j(t) + \lambda B_j(t) = 0$ for $0<t\le T$, $B_j(0)=w_0$. Then $B_j(t) = w_0 E_{\alpha_j,1}(-\lambda t^{\alpha_j}/\underline q_j) $ by \cite[Example 3.1]{JinBook}; this function is completely monotonic \cite[Theorem 3.5]{JinBook}, which says  in particular that $B_j(t)\ge 0$ and $B_j'(t)\le 0$.  (In the case $\underline q_j=0$ one takes $B_j\equiv 0$.)
Consequently $(w-B_j)(0)\ge 0$ and
\[
\left(D_t^{\bar\alpha}+\lambda\right)(w-B_j)(t) = v(t) - \left[q_j(t)-\underline q_j \right]D _t ^{\alpha_j} B_j(t)
	-\sum_{i\ne j} q_i D _t ^{\alpha_i} B_j(t)  \ge 0\ \text{ for } t>0.
\]
Lemma~\ref{lem:v0w0}  now yields $w(t)\ge B_j(t)$ for all $t\in[0,T]$, which implies the desired result since  $j\in \{1,\dots,\ell\}$ was arbitrary.
\end{proof}

In the case $\ell=1$, constant $q_1>0$,  and $v\equiv 0$, the bound of the theorem is sharp.

\tcb{In the conclusion \eqref{wtEa} of Theorem~\ref{thm:wlowerbound}, the value of~$j$ such that $E_{\alpha_j,1}(-\lambda t^{\alpha_j}/\underline q_j)$ is the dominant term may change as $t$ varies. This phenomenon is illustrated in Figure~\ref{fig:Egraphs}, where $\ell=3$, $w_0=\lambda = \underline q_j =1$ for each~$j$, and $\alpha_j \in \{1, 0.7, 0.3\}$;  one sees that Theorem~\ref{thm:wlowerbound} yields $w(t)\ge E_{1,1}(-t)$ for $0\le t < 0.7$ (approx.) but $w(t)\ge E_{0.3,1}(-t^{0.3})$ for $0.7< t\le 2$.}

  \begin{figure}[h!]
\begin{center}
\hspace*{-0.3cm}\includegraphics[height=0.55\textwidth]{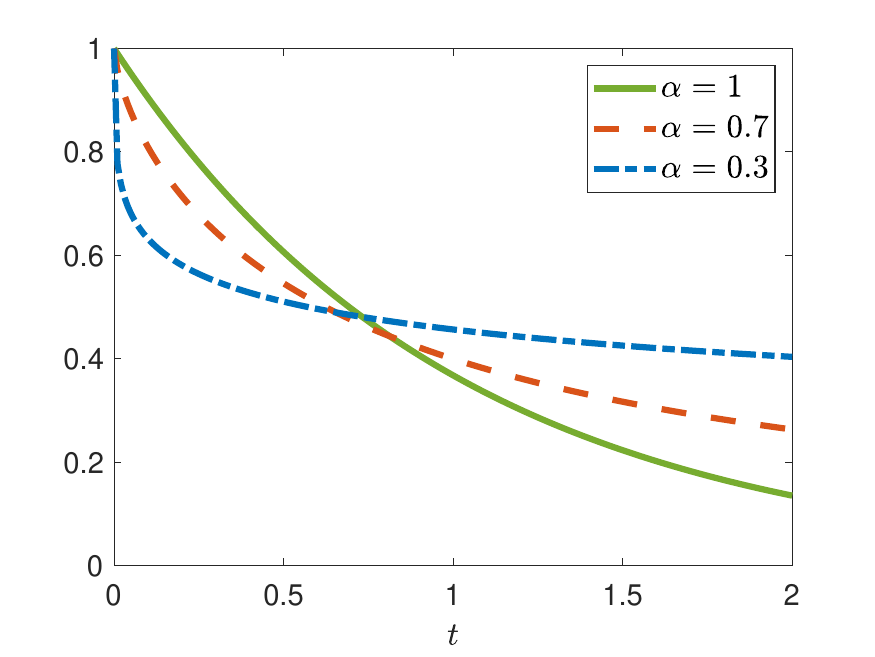}\hspace*{-0.0cm}
\end{center}
 \caption{\label{fig:Egraphs}\it\small
\tcb{Illustration of result of Theorem~\ref{thm:wlowerbound}: graphs of $E_{\alpha,1}(-t^{\alpha})$ for $\alpha =1, 0.7, 0.3$}}
 \end{figure}

In the next remark we discuss the behaviour of $w'(t)$ as $t\to 0^+$.

\begin{remark}\label{rem:w'(t)}
Assume the hypotheses of Theorem~\ref{thm:wlowerbound} and that $v\in C^1[0,T]$.
Assume also that $q_1(0)\ne 0$ and $q_1\in C^1[0,T]$; without loss of generality we can take $q_1(0)=1$.
Set $\phi(t):=t^{\alpha_1}/\Gamma(1+\alpha_1)$. Then
\[
\left(D_t^{\bar\alpha}+\lambda\right)\phi(t) = q_1(t) +\sum_{i=2}^\ell \frac{q_i(t) t^{\alpha_1-\alpha_i}}{\Gamma(1+\alpha_1-\alpha_i)}
		+ \frac{\lambda t^{\alpha_1}}{\Gamma(1+\alpha_1)}
		=1+O(t^{\alpha_1-\alpha_2}),
\]
so $\left(D_t^{\bar\alpha}+\lambda\right)[1-\lambda\phi(t)]=O(t^{\alpha_1-\alpha_2})$.

Let $\delta$ be a nonzero constant. Set $\widetilde w_{\delta}(t):=w_0\,[1-\lambda\phi(t)]+[v(0)+\delta]\,\phi(t)$. Then $\widetilde w_{\delta}(0) = w_0$ and
\[
(D_t^{\bar\alpha}+\lambda)\,\widetilde w_{\delta}(t) = O(t^{\alpha_1-\alpha_2}) + [v(0)+\delta]\left[ 1+ O(t^{\alpha_1-\alpha_2})\right]
	= v(t)+\delta+O(t^{\alpha_1-\alpha_2})
\]
since $v(t)=v(0)+O(t)$. Now choose $\delta$ to be a small positive constant. Then choose $\epsilon>0$ such that for $t\in(0,\epsilon)$ one has $ \vert O(t^{\alpha_1-\alpha_2}) \vert \le \delta$ in the previous equation. Now the comparison principle (Lemma~\ref{lem:v0w0}) yields $\widetilde w_{-\delta}(t)\le w(t)\le \widetilde w_{\delta} (t)$ for $0\le t\le\epsilon$.
That is,
\[
[v(0)-\lambda w_0-\delta]\,\phi(t)\le w(t)-w_0\le[v(0)-\lambda w_0+\delta]\,\phi(t) \ \text{ for } 0\le t\le \epsilon.
\]
Hence $(w(t)-w(0))/t\approx [v(0)-\lambda w_0]\phi(t)$ as $t\to 0$, and assuming that $v(0)\neq\lambda w_0$, one has
 $(w(t)-w(0))/t = O(t^{\alpha_1-1})$ as $t\to 0$.

 Thus, when $\alpha_1<1$  we expect that $w'(t)\to -\infty$ as $t\to 0^+$, but if $\alpha_1=1$, then the behaviour of the solution is quite different: we expect that $w'(t)$ remains bounded as $t\to 0^+$. This behaviour when $\alpha_1=1$  concurs with the existence result for~\eqref{IVP} that we shall prove rigorously in Lemma~\ref{lem:special}.
\end{remark}

\noindent{\it Notation.}
From Lemma~\ref{lem:v0w0} it follows that any solution of \eqref{IVP} is unique.
We shall use the  notation
$\left(D_t^{\bar\alpha}+\lambda\right)^{-1}\!v$
for this unique solution.
\smallskip

When the $q_i$ are positive constants, then the solution of \eqref{IVP} exists and can be written in an explicit
form; this will be seen in  Section~\ref{ssec_const}.
For the general case of variable $q_i(t)$, existence of the solution to~\eqref{IVP} seems reasonable but it does remain an open question; nevertheless, almost all of our analysis does not require this existence result --- the only exception is Corollary~\ref{cor1_L2}.

\subsection{Solution of \eqref{IVP} for constant-coefficient $D_t^{\bar\alpha}$}\label{ssec_const}
Throughout Section~\ref{ssec_const}, let all $q_i$ in \eqref{DD_t_def} be positive constants.
Without loss of generality, we assume that $q_1=1$.

Then the  structure of the solution of \eqref{IVP} is intimately related to the following multinomial Mittag-Leffer function, which is a generalisation of the two-parameter Mittag-Leffler function $E_{\alpha,\beta}(s)$.

\begin{definition}\label{def:multiML}\cite{LLY15,LG99}
Let $\beta_0 \in (0,2)$. For $j=1,\dots,\ell$,  let $0<\beta_j\le 1, \ s_j \in \mathbb{R}$ and  $k_j\in \mathbb{N}_0$. Then the multinomial Mittag-Leffler function  is defined by
\begin{equation}\label{ML2}
E_{(\beta_1,\dots,\beta_\ell),\beta_0}(s_1,\dots,s_\ell):=\sum_{k=0}^{\infty}\sum_{k_1+\dots+k_\ell \atop =k}^{}\frac{(k;k_1,\dots,k_\ell)\prod_{j=1}^{\ell}s_j^{k_j}}{\Gamma(\beta_0+\sum_{j=1}^{\ell}\beta_jk_j)}\,,
\end{equation}
where  the multinomial coefficient
\begin{equation*}
(k;k_1,\dots,k_\ell):= \frac{k!}{k_1!\cdots k_\ell!} \ \text{ with }   k=\sum_{j=1}^{\ell}k_j.
\end{equation*}
\end{definition}


\begin{remark}\label{rem:Esymm}
The symmetry of Definition~\ref{def:multiML} implies that the value of $E_{(\beta_1,\dots,\beta_\ell),\beta_0}(s_1,\dots,s_\ell)$ remains unaltered if we perform any permutation of $(\beta_1,\dots,\beta_\ell)$, provided that we also perform the same permutation of $(s_1,\dots,s_\ell)$. In particular one has
$E_{(\beta_1,\dots,\beta_\ell),\beta_0}(s_1,\dots,s_\ell) = E_{(\beta_\ell,\dots,\beta_1),\beta_0}(s_\ell,\dots,s_1)$.
\end{remark}

We shall also use the more succinct notation of~\cite[eq.(2.4))]{Baz21}:
\beq\label{Baz}
\FF_{(\mu_1, \mu_2,\dots,\mu_m),\beta}(t; a_1, a_2, \dots, a_m) :=
t^{\beta-1}E_{(\mu_1, \mu_2,\dots,\mu_m),\beta}(-a_1t^{\mu_1},-a_2t^{\mu_2},\dots,-a_mt^{\mu_m})
\eeq
for $t>0$, any positive integer~$m$,  $\beta \in (-\infty,2)$,  $0 <\mu_j <1$ for each~$j$, and any real constants $a_1,\dots,a_m$.

\begin{lemma}\label{lem:beta<1}
Suppose that $0\le \mu_m <\dots < \mu_1 \le \beta \le 1$ and $a_j>0$ for $j=1,\dots, m$. Then
$E_{(\mu_1, \mu_2,\dots,\mu_m),\beta}(-a_1t^{\mu_1},-a_2t^{\mu_2},\dots,-a_mt^{\mu_m})\ge 0 $ for all $t> 0$.
\end{lemma}
\begin{proof}
Taking $\delta=1$ in \cite[Theorem 3.2]{Baz21} shows that  $t\mapsto \FF_{(\mu_1, \mu_2,\dots,\mu_m),\beta}(t; a_1, a_2, \dots, a_m)$ is a completely monotone function, which implies that $\FF_{(\mu_1, \mu_2,\dots,\mu_m),\beta}(t; a_1, a_2, \dots, a_m) \ge 0$. The desired result now follows from~\eqref{Baz}.
\end{proof}

If $w_0=0$, then \cite[Theorem 4.1]{LG99}  gives the solution of \eqref{IVP} as
\begin{align}
&w(t)=\left(D_t^{\bar\alpha}+\lambda\right)^{-1}v \notag\\
&= \int_{s=0}^t  s^{\alpha_1-1}E_{(\alpha_1-\alpha_2,\alpha_1-\alpha_3,\dots,\alpha_1-\alpha_\ell,\alpha_1),\alpha_1}
	(-q_2s^{\alpha_1-\alpha_2},-q_3s^{\alpha_1-\alpha_3},\dots,-q_\ell s^{\alpha_1-\alpha_\ell},-\lambda s^{\alpha_1}) \notag\\
	&\hspace{10cm} v(t-s)\,ds \notag\\
	&= \int_{s=0}^t  s^{\alpha_1-1}E_{(\alpha_1,\alpha_1-\alpha_\ell, \dots,\alpha_1-\alpha_3,\alpha_1-\alpha_2),\alpha_1}	(-\lambda s^{\alpha_1},-q_\ell s^{\alpha_1-\alpha_\ell}, \dots, -q_3s^{\alpha_1-\alpha_3},-q_2s^{\alpha_1-\alpha_2}) \notag\\
	&\hspace{10cm} v(t-s)\,ds \notag\\
	 &= \int_{s=0}^t \FF_{(\alpha_1,\alpha_1-\alpha_\ell, \dots,\alpha_1-\alpha_3,\alpha_1-\alpha_2),\alpha_1}
			(s; \lambda, q_\ell, \dots, q_3, q_2) \,v(t-s)\,ds, 	 \label{w1}
\end{align}
where we used  Remark~\ref{rem:Esymm} and \eqref{Baz}. The formula \eqref{w1} is the multiterm generalisation of \cite[eq.(2.1)]{KopAML22}.

\begin{remark}\label{rem:IVPconstsoln}
It is easy to see that $\left(D_t^{\bar\alpha}+\lambda\right)[1]=\lambda$. Hence if  $w_0\ne 0$, then the initial-value problem ~\eqref{IVP}, with positive constant coefficients $q_i$, has the unique solution
\[
w(t) = w_0 + \left(D_t^{\bar\alpha}+\lambda\right)^{-1}\left[v(t) -\lambda w_0\right]\quad\text{for }t\in[0,T].
\]
Furthermore, an explicit solution representation for $ \left(D_t^{\bar\alpha}+\lambda\right)^{-1}[v-\lambda w_0]$ is provided by~\eqref{w1} with $v(t-s)$ replaced by $v(t-s)-\lambda w_0$.
\end{remark}

\subsection{The special case $\alpha_1=1$ and $q_1(t)>0$}\label{sec:special}
In this subsection we consider the special case where $\alpha_1=1$ and $q_1(t)>0$ for all $t\in[0,T]$. In this setting we are able to prove existence of a solution $w$ to the variable-coefficient initial-value problem~\eqref{IVP} , and moreover this solution lies in $C^1[0,T]$.

\begin{lemma}\label{lem:special}
Assume that $\alpha_1=1$ and $q_1(t)>0$ for all $t\in[0,T]$, with $v,q_i\in C[0,T]$ for all~$i$. Then  the initial-value problem \eqref{IVP} has a solution $w\in C^1[0,T]$, and this solution is unique.
\end{lemma}
\begin{proof}
Lemma~\ref{lem:v0w0} implies that any solution of \eqref{IVP} is unique.
To show existence of a solution we assume without loss of generality that $q_1(t)\equiv 1$ for $t\in[0,T]$, since one can divide \eqref{IVP} by~$q_1(t)$. Set $\phi(t) = w'(t)$. Using the definition~\eqref{CaputoEquiv},  write \eqref{IVP} as
\begin{equation}\label{phi1}
\phi(t) + \sum_{i=2}^\ell \int_{0}^t(t-s)^{-\alpha_i}q_i(t)   \phi(s)\, ds + \lambda \int_0^t \phi(s)\,ds = v(t) - \lambda w_0 \ \text{ for }t\in [0,T].
\end{equation}
This is a Volterra integral equation of the second kind  in the unknown function~$\phi$.
Observe first that any solution of~\eqref{phi1} in $C[0,T]$ must be unique, because two distinct solutions $\phi_1, \phi_2$ would yield two distinct solutions  $w_i(t) := w_0 + \int_0^t\phi_i(s)\,ds$ ($i=1,2$) of~\eqref{IVP}.
It is well known (see, e.g., \cite[Appendix A.2.2]{Bru17}) that each of the operators
\[
\phi(t) \mapsto  \int_{0}^t(t-s)^{-\alpha_i}q_i(t)   \phi(s)\, ds \quad \text{ and } \quad \phi(t)\mapsto  \int_0^t \phi(s)\,ds
\]
is a compact operator from the Banach space $(C[0,T], \|\cdot\|_\infty)$ to itself, and a finite sum of compact operators is also a compact operator, and  we saw already that any solution of~\eqref{phi1}  is unique;
 thus we can apply the Fredholm Alternative Theorem \cite[Theorem A.2.17]{Bru17} to conclude that~\eqref{phi1} has a solution $\phi\in C[0,T]$. Hence \eqref{IVP} has the solution $w(t) := w_0 + \int_0^t\phi(s)\,ds$, and this solution clearly lies in $C^1[0,T]$.
\end{proof}

\section{$L_2(\Omega)$ a posteriori error estimates}\label{sec:L2apost}
Let $u_h$ be our  \tcb{approximate solution.} We assume throughout our analysis that $u_h(\cdot,0)=u(\cdot,0)$ on~$\Omega$ and  $u_h(x,t)=u(x,t)$ for $x\in\pt\Omega$ and $t>0$. For the case $u_h(\cdot,0)\ne u(\cdot,0)$, see \cite[Corollary 2.5]{KopAML22}.

\begin{lemma}\label{lem_aux}
Suppose that $r(\cdot,0)=0$ and
$r\in L_{\infty}(0,T;\,L_2(\Omega)) \cap W^{1,\infty}(\epsilon,T;\,L_2(\Omega))$ for each $\epsilon\in (0,T]$,
Then
\[
\left\langle  \sum_{i=1}^{\ell} q_i(t)\, D _t ^{\alpha_i}r(\cdot,t),\,r(\cdot,t) \right\rangle \ge \left(\sum_{i=1}^{\ell} q_i(t)\, D _t ^{\alpha_i}\|r(\cdot,t)\|\right) \|r(\cdot,t)\|\qquad\mbox{for}\;\;t>0.
\]
\end{lemma}
\begin{proof}
One can use the same proof as for \cite[Lemma 2.8]{KopAML22}, based on the reformulation \eqref{altCaputo} and recalling Remark~\ref{rem:limt0}.
\end{proof}

Define the residual
\[
R_h(x,t) := \left(\sum_{i=1}^{\ell} q_i(t)\, D _t ^{\alpha_i} + \LL\right) u_h(x,t)-f(x,t)
	\ \text{ for all } (x,t)\in\Omega\times (0,T].
\]

 \begin{theorem}\label{the_L2}
In \eqref{pde} assume that  $\langle \LL r,r\rangle\ge \lambda\|r\|^2$ for all $r\in H_0^1(\Omega)$, where $\lambda \ge 0$ is some constant. Suppose that a unique solution $u$ of \eqref{problem} and its approximation $u_h$ are
in $C([0,T];\,L_2(\Omega)) \cap W^{1,\infty}(\epsilon,T;\,L_2(\Omega))$ for each $\epsilon\in (0,T]$,
and also in $H^1_0(\Omega)$ for each $t>0$.
Suppose also that
\beq\label{res_bound}
\|R_h(\cdot,t)\|\le \left( \sum_{i=1}^{\ell} q_i(t) D _t ^{\alpha_i} + \lambda\right) \EE(t)\quad \forall\,t>0
\eeq
for some barrier function $\EE$ that satisfies the regularity condition \eqref{ycondition}, with moreover $\EE(t)\ge 0\ \forall\,t\ge0$.
Then $\|(u_h-u)(\cdot,t)\|\le \EE(t)$ $\forall\,t\ge0$.%
\end{theorem}

\begin{proof}
(This is similar to the proof of \cite[Theorem 2.2 and Corollary 2.3]{KopAML22}.)

Set $e: = u_h-u$, so $e(\cdot,0)=0,\ e(x,t)=0$ for $x\in\pt\Omega$, and $\left(D _t ^{\bar\alpha} + \LL\right)e(\cdot,t) = R_h(\cdot,t)$ for~$t>0$. Multiply this equation by $e(\cdot,t)$ then integrate over~$\Omega$;  invoking Lemma~\ref{lem_aux} and
$\langle \LL v,v\rangle\ge \lambda\|v\|^2$, we get
\beq\label{e_eq}
(D_t^{\bar\alpha}+\lambda) \|e(\cdot,t)\|\le \|R_h(\cdot,t)\| \quad\text{for } t>0.
\eeq
Combining this with our hypothesis \eqref{res_bound}, one has $(D_t^{\bar\alpha}+\lambda)(\EE-\|e(\cdot,t)\|)\ge 0$. Now an application of Lemma~\ref{lem:v0w0} yields $\|(u_h-u)(\cdot,t)\|\le \EE(t)$ $\forall\,t\ge0$.
\end{proof}

In Theorem~\ref{the_L2}, one can replace the condition $\EE\in C[0,T]$ of \eqref{ycondition}  by $\lim_{t\to 0^+}\EE(t)\ge 0$ exists; see~Remark~\ref{rem:limt0} and Corollary~\ref{cor:wnotcont}.

Note that the  proof of Theorem~\ref{the_L2} did not require existence of a solution of \eqref{IVP}, which we have proved only for the constant-coefficient case of Section~\ref{ssec_const} and the case $\alpha_1=1$ and $q_1(t)>0$ of Section~\ref{sec:special}.

The next corollary presents a possible choice of $\EE(t)$ to use in~\eqref{res_bound}.

\begin{corollary}\label{cor:L2boundEE0}
Assume the hypotheses of Theorem~\ref{the_L2}. Then the error $e=u_h-u$ satisfies
\begin{equation}\label{R0bound}
\|(u_h-u)(\cdot,t)\|\le \sup_{0<s\le t}\!\left\{ \frac{ \|R_h(\cdot,s)\|}{\RR_0(s) }\right\},
\ \text{ where }
\RR_0(t) := \lambda +  \sum_{i=j}^{\ell} q_i(t)\, t^{-\alpha_i}/\Gamma(1-\alpha_i)
\end{equation}
where $j=1$ if $\alpha_1<1$ and $j=2$ if $\alpha_1=1$.
\end{corollary}
\begin{proof}
(The proof is similar to part of the proof of \cite[Corollary 2.4]{KopAML22}.)

Set $\kappa = \sup_{0<s\le t}\!\left\{ \|R_h(\cdot,s)\|/\RR_0(s) \right\}$. If $\kappa=\infty$ the result is trivial, so assume that $0\le \kappa \in \bR$. Define the barrier function $\EE_0(t)$ by  $\EE_0(t):=1$ for $t>0$ and $\EE_0(0):=0$.
Note that $\EE_0$ satisfies the conditions of Theorem~\ref{the_L2}.
From~\eqref{altCaputo}  (see also \cite[Remark 2.9]{KopAML22})
one has $D _t ^{\alpha_i}  \EE_0(t) = t^{-\alpha_i}/\Gamma(1-\alpha_i)$ for $t>0$ if $\alpha_i<1$, while $D_t^{\alpha_1}\EE_0(t)=0$ for $t>0$ if $\alpha_1=1$,  so
 $ \left(D _t ^{\bar\alpha} +\lambda \right)\kappa \EE_0(t) = \kappa\RR_0(t)$ in all cases. Thus we can apply Theorem~\ref{the_L2} with $\EE= \kappa\EE_0$ to finish the proof.
\end{proof}

We shall present a second possible choice of $\EE(t)$ after we list some properties of the hypergeometric function
$_2F_1(\alpha_i,-\beta\,;\,\alpha_1\,;\,s)$ that is discussed in \cite[Section 15]{AS64} and \cite{DLMF15}.

Set  $\beta:=1-\alpha_1$. Then
$\frac{d}{ds}\bigl(s^{-\beta}{}_2F_1(\alpha_i,-\beta\,;\,\alpha_1\,;\,s)\bigr)
=-\beta s^{-\beta-1}{}_2F_1(\alpha_i,-\beta\,;\,-\beta\,;\,s)$ \cite[Section 15.2.4]{AS64} \cite[item 15.5.4]{DLMF15},
while by \cite[Section 15.1.8]{AS64} \cite[item 15.4.6]{DLMF15}
 one gets ${}_2F_1(\alpha_i,-\beta\,;\,-\beta\,;\,s)=(1-s)^{-\alpha_i}$.
Hence
\begin{equation}\label{2F1A}
\frac{d}{ds}\Bigl(s^{-\beta}{}_2F_1(\alpha_i,-\beta\,;\,\alpha_1\,;\,s)\Bigr)
=-\beta\,s^{-\beta-1}(1-s)^{-\alpha_i}.
\end{equation}
Furthermore,
\begin{equation}\label{2F1B}
{}_2F_1(\alpha_i,-\beta\,;\,\alpha_1\,;\,1)
= \frac{\Gamma(\alpha_1)\Gamma(1-\alpha_i)}{\Gamma(\alpha_1-\alpha_i)}
\end{equation}
by \cite[Section 15.1.20]{AS64}\cite[item 15.4.20]{DLMF15}.

\begin{corollary}\label{cor:L2boundEE1}
Assume the hypotheses of Theorem~\ref{the_L2} and that $\alpha_1<1$.
Then the error $e=u_h-u$ satisfies
\begin{equation}\label{R1bound}
\|(u_h-u)(\cdot,t)\|\le t^{\alpha_1-1}\sup_{0<s\le t}\!\left\{ \frac{ \|R_h(\cdot,s)\|}{\RR_1(s) }\right\},
\end{equation}
where
\begin{align}\notag
 \RR_1(t) &  := \lambda\EE_1(t) +  \tau^{-\beta}\sum_{i=1}^{\ell} \frac{q_i(t)\, t^{-\alpha_i}}{\Gamma(1-\alpha_i)}\bigl[1-\rho_i(\hat\tau)\bigr],
 \qquad
 \rho_i(s):=0\;\;\mbox{for}\;s\ge1,
\\\label{rho_i}
  \rho_i(s)& :=
  {}_2F_1(\alpha_i,-\beta\,;\,\alpha_1\,;\,s)-
  {\textstyle\frac{\Gamma(\alpha_1)\Gamma(1-\alpha_i)}{\Gamma(\alpha_1-\alpha_i)}}\,s^{\beta}
  \le (1-s)^{1-\alpha_i}<1
  \;\;\mbox{for}\;s\in(0,1),
\end{align}
with $\beta:=1-\alpha_1$, $ \hat\tau := \tau/t$ and $\EE_1(t):= (\max\{\tau, t\})^{\alpha_1-1}$ for $t>0, \ \EE_1(0):=0$.
Here $\tau>0$ is an  arbitrary  user-chosen parameter.

Furthermore, in \eqref{R1bound} one can replace $t^{\alpha_1-1}$ by $\EE_1(t)$ if desired.
\end{corollary}
\begin{proof}
Set $\kappa = \sup_{0<s\le t}\!\left\{ \|R_h(\cdot,s)\|/\RR_1(s) \right\}$. If $\kappa=\infty$ the result is trivial, so assume that $0\le \kappa \in \bR$.
Observe that
$\EE_1(t)=\tau^{-\beta}\EE_0(t) - (\tau^{-\beta}-t^{-\beta})^+$, where $\EE_0$ was defined in the proof of Corollary~\ref{cor:L2boundEE0}.  From~\eqref{altCaputo}  one has
$D _t ^{\alpha_i}  \EE_0(t) = t^{-\alpha_i}/\Gamma(1-\alpha_i)$ for $t>0$ and $i=1,\dots,\ell$, so for $t\le \tau$ (i.e., $\hat\tau\ge 1$) we get
$ \left( D _t ^{\bar\alpha} +\lambda \right)\kappa \EE_1(t) = \kappa\RR_1(t)$ since in $\RR_1$ one has $\rho_i(\hat\tau)=0$ for all~$i$.

For $t>\tau$, since $\pt_s (\tau^{-\beta}-s^{-\beta})^+ =-\pt_s(s^{-\beta})
=\beta s^{-\beta-1} $, for $i=1,2,\dots,\ell$ we have
\begin{align*}
\Gamma(1-\alpha_i)\,D^{\alpha_i}_t \EE_1(t)
	&= \tau^{-\beta}t^{-\alpha_i} -\beta\int_{\tau}^t\! s^{-\beta-1}(t-s)^{-\alpha_i}\,ds \\
&=
\tau^{-\beta}t^{-\alpha_i} -\beta t^{-\beta-\alpha_i}\int_{\hat\tau}^1\! \hat s^{-\beta-1}(1-\hat s)^{-\alpha_i}\,d\hat s\\
&=
\tau^{-\beta}t^{-\alpha_i} - t^{-\beta-\alpha_i}
\Bigl(\hat\tau^{-\beta}\,\rho_i(\hat\tau)\Bigr),
\end{align*}
from \eqref{2F1A} and \eqref{rho_i}. But $t^{-\beta-\alpha_i}\,\hat\tau^{-\beta}=\tau^{-\beta}t^{-\alpha_i}$,
so $\Gamma(1-\alpha_i)\,D^{\alpha_i}_t \EE_1(t)=\tau^{-\beta}t^{-\alpha_i}[1-\rho_i(\hat\tau)]$.
Hence, $ \left( D _t ^{\bar\alpha} +\lambda \right)\kappa \EE_1(t) \ge \kappa\RR_1(t)$ for $t> \tau$.

For the bound on $\rho_i$ in \eqref{rho_i},
the above argument shows that
\begin{align*}
\hat\tau^{-\beta}\,\rho_i(\hat\tau)&=\beta\int_{\hat\tau}^1\! \hat s^{-\beta-1}(1-\hat s)^{-\alpha_i}\,d\hat s\\
	&\le \beta (1-\hat\tau)^{\alpha_1-\alpha_i} \int_{\hat\tau}^1\! \hat s^{-\beta-1}(1-\hat s)^{-\alpha_1}\,d\hat s \\	
&=(1-\hat\tau)^{\alpha_1-\alpha_i}\,\Bigl(\hat\tau^{-\beta}(1-\hat\tau)^{1-\alpha_1}\Bigr)
=\hat\tau^{-\beta}\,(1-\hat\tau)^{\alpha_i},
\end{align*}
where we also used $(1-\hat s)^{-\alpha_i}/(1-\hat s)^{-\alpha_1}\le (1-\hat \tau)^{\alpha_1-\alpha_i}$ as $\alpha_1\ge\alpha_i$.
Hence, $\rho_i(\hat\tau)\le (1-\hat\tau)^{\alpha_i}$, as desired.

We can now apply Theorem~\ref{the_L2} with $\EE= \kappa\EE_1$ to obtain the bound
\[
\|(u_h-u)(\cdot,t)\|\le \EE_1(t) \sup_{0<s\le t}\!\left\{ \frac{ \|R_h(\cdot,s)\|}{\RR_1(s) }\right\};
\]
then $\EE_1(t) \le t^{\alpha_1-1}$ completes the proof.
\end{proof}

One could extend the proof of  Corollary~\ref{cor:L2boundEE1} to include the case $\alpha=1$, but in this case the result becomes the same as that of Corollary~\ref{cor:L2boundEE0}.

Note that \eqref{rho_i} for $i=1$ simplifies to
$\rho_1(s)=\left((1-s)^+\right)^{1-\alpha_1}$
because of \cite[Section 15.1.8]{AS64} \cite[item 15.4.6]{DLMF15} and $\{\Gamma(\alpha_1-\alpha_i)\}^{-1}=0$.

Finally, we give a general result that relates $\|(u_h-u)(\cdot,t)\|$ to $\|R_h(\cdot, t)\|$ without involving any  barrier function --- but this result, unlike Corollaries~\ref{cor:L2boundEE0} and~\ref{cor:L2boundEE1}, depends on the assumption that $\left(D _t ^{\bar\alpha} +\lambda \right)^{-1}\|R_h(\cdot, t)\|$ exists.

\begin{corollary}\label{cor1_L2}
Assume the hypotheses of Theorem~\ref{the_L2}. Recall the definition of $\left(D _t ^{\bar\alpha} +\lambda \right)^{-1}$ in Section~\ref{sec:nonnegative}. If $\left(D _t ^{\bar\alpha} +\lambda \right)^{-1}\|R_h(\cdot, t)\|$ exists,
then
\beq\label{L2_error}
\|(u_h-u)(\cdot,t)\|\le \left( \sum_{i=1}^{\ell} q_i(t) D _t ^{\alpha_i} +\lambda \right)^{-1}\|R_h(\cdot, t)\|\qquad\mbox{for}\;\;t>0.
\eeq
\end{corollary}
\begin{proof}
Set $\EE(t):= (D_t^{\bar\alpha}+\lambda)^{-1}\|R_h(\cdot, t)\|$. Then $\EE(0)=0$ and $(D_t^{\bar\alpha}+\lambda)\EE(t)=\|R_h(\cdot, t)\|$  imply $\EE(t)\ge 0$ by Lemma~\ref{lem:v0w0}. Thus we can invoke Theorem~\ref{the_L2} to get \eqref{L2_error}.
\end{proof}

\section{$L_\infty(\Omega)$ a posteriori error estimates}\label{sec:Linf}
Throughout Section~\ref{sec:Linf}, let
 $\LL u := \sum_{k=1}^d \left\{ a_k(x)\,\pt^2_{x_k}\!u  + b_k(x)\, \pt_{x_k}\!u \right\}+c(x)\, u$
 in \eqref{problem},
with sufficiently smooth coefficients $\{a_k\}$, $\{b_k\}$ and $c$ in $C(\bar\Omega)$. Assume also that for each $k$ one has $a_k>0$ in $\bar\Omega$, and that $c\ge \lambda\ge 0$.

The condition $\langle \LL v,v\rangle\ge \lambda\|v\|^2$ is not required in this section.

\begin{lemma}[Comparison principle for the initial-boundary value problem]\label{lem_max_pr}
Suppose that
\begin{equation}\label{comp}
\sum_{i=1}^{\ell}\bigl[q_i(t)\, D _t ^{\alpha_i} v(x,t)\bigr]+ \LL v(x,t) \ge 0\ \text{ for }(x,t)\in \Omega\times(0,T],
\end{equation}
where $v(\cdot, t)\in C^2(\Omega)$ for each~$t>0$, and for each $x\in\Omega$ we have $v(x,\cdot)\in W^{1,\infty}(\epsilon,t)$ for all $\epsilon, t$ satisfying $0<\epsilon<t\le T$, and $\lim_{t\to 0^+}v(x,t)\ge 0$ exists.
In \eqref{comp} define $D_t^{\alpha_i} v(x,\cdot)$ for each $x\in\Omega$ by~\eqref{altCaputo} if $\alpha_i<1$. If $\alpha_1=1$,  suppose also that $v_t(x,\cdot)$ (for each $x\in\Omega$) may have jump discontinuities but is a left-continuous function on $(0,T]$. Assume that $v(x,0)\ge 0$ for $x\in \Omega$ and that $v(x, t)\ge 0$ for $x\in\pt\Omega$ and $0\le t\le T$. Then $v(x,t)\ge 0$ for all $(x,t)\in \Omega\times (0,T]$.
\end{lemma}
\begin{proof}
Imitate the argument of Corollary~\ref{cor:wnotcont}, with the extra detail that $\LL v(x_0,t_0)\le 0$ at any point $(x_0, t_0)\in \Omega\times (0,T]$ where $v(x,t)$ attains a negative minimum.
\end{proof}

A result similar to Lemma~\ref{lem_max_pr} is proved in \cite[Theorem~2]{Luch11} under the stronger hypothesis  that $v(x,\cdot)\in C^1(0,T]\cap W^{1,1}(0,T)$ for each $x\in\Omega$. See also \cite[Lemma 3.1]{BHY15}.

\begin{theorem}
Assume that a unique solution $u$ of \eqref{problem} and its approximation $u_h$ each satisfy the regularity hypotheses imposed on $v$ in Lemma~\ref{lem_max_pr}.
Then the error bounds of Theorem~\ref{the_L2} 
and Corollaries~\ref{cor1_L2}, \ref{cor:L2boundEE0} and~\ref{cor:L2boundEE1} remain true with $\|\cdot\|=\|\cdot\|_{L_2(\Omega)}$
replaced by $\|\cdot\|_{L_\infty(\Omega)}$.
\end{theorem}
\begin{proof}
Note that $\LL\EE(t) = c\EE(t) \ge \lambda \EE(t)$ for~$t>0$.
Consider first Theorem~\ref{the_L2}, whose hypothesis now becomes $\|R_h(\cdot,t)\|_{L_\infty(\Omega)}\le (D_t^{\bar\alpha}+\lambda)\EE(t)$ for~$t>0$. But $(D_t^{\bar\alpha}+\lambda)\EE(t)\le (D_t^{\bar\alpha}+\LL)\EE(t)$ and $R_h(x,t) = (D_t^{\bar\alpha}+\LL)(u_h-u)(x,t)$, so we have $\vert(D_t^{\bar\alpha}+\LL)(u_h-u)(x,t)\vert\le (D_t^{\bar\alpha}+\LL)\EE(t)$ for $x\in\Omega$ and $t>0$. Thus one can invoke Lemma~\ref{lem_max_pr} to get  $\vert(u_h-u)(x,t)\vert\le \EE(t)$ on $\Omega\times[0,T]$, i.e., Theorem~\ref{the_L2} is valid in the $L_\infty(\Omega)$ setting.

We can now deduce  $L_\infty(\Omega)$ variants of the other results.
To get the new Corollary~\ref{cor1_L2}, take $\EE(t):=(D_t^\alpha+\lambda)^{-1}\|R_h(\cdot,t)\|_{L_\infty(\Omega)}$ in the new Theorem~\ref{the_L2}. For the new Corollaries~\ref{cor:L2boundEE0} and~\ref{cor:L2boundEE1}, use their old proofs with $\|R_h\|$ replaced by $\|R_h\|_{L_\infty(\Omega)}$ and appeal to the new Theorem~\ref{the_L2}.
\end{proof}

\section{Application to the L1 method. Numerical experiments}\label{sec:L1method}

In this section we examine in detail the practical application of our a posteriori analysis to the well-known L1 discretisation of each fractional derivative  $ D _t ^{\alpha_i}$. Other discretisations will be discussed in a future paper~\cite{FraKopInPrep}.

Given an arbitrary temporal mesh $\{t_j\}_{j=0}^M$ on $[0,T]$, let $\{u_h^j\}_{j=0}^M$ be the semi-discrete approximation for \eqref{problem} obtained using the
popular L1 method \cite{StyL1survey}.
Then its standard Lagrange piecewise-linear-in-time interpolant $u_h$, defined on $\bar\Omega\times[0,T]$,
 satisfies\vspace{-0.1cm}
\beq\label{L1method}
\Bigl(\sum_{i=1}^{\ell}q_i(t_j)\, D _t ^{\alpha_i}+\LL\Bigr)u_h(x, t_j)=f(x, t_j)\qquad \mbox{for}\;\; x\in\Omega,\;\;  j=1\ldots, M,\vspace{-0.1cm}
\eeq
subject to $u_h^0:=u_0$ and $u_h=0$ on $\pt \Omega$.
In the case of $\alpha_1=1$, the term $D _t ^{\alpha_1}u_h(x, t_j)=[u_h(x, t_j)-u_h(x, t_{j-1})]/(t_j-t_{j-1})$, which corresponds to $\pt_t u_h$ treated as a left-continuous function in time.

First, consider the case $\alpha_1<1$.
For
the residual of $u_h$ one immediately gets $R_h(\cdot, t_j)=0$ for $j\ge 1$,
i.e., the residual is a non-symmetric bubble on each $(t_{j-1},t_j)$ for $j>1$.
Hence, for the piecewise-linear interpolant $R_h^I$ of  $R_h$ one has $R_h^I=0$ for $t\ge t_1$,
and, more generally,
$R_h^I=[\LL u_0-f(\cdot, 0)](1-t/t_1)^+$ for $t>0$
(where we used $R_h(\cdot,0)=\LL u_0-f(\cdot,0)$ because  $D_t^{\alpha_i} u_h^0(\cdot, 0)=0$).
Finally, note that  $R_h-R_h^I=(D_t^{\bar\alpha} u_h-f)-(D_t^{\bar\alpha} u_h-f)^I$
since $(\LL u_h)^I
=\LL u_h$.
%
In other words, one can compute $R_h$ by sampling, using parallel/vector evaluations, without a direct application of $\LL$ to $\{u^j_h\}$.%


Next, consider the case $\alpha_1=1$. Then $D _t ^{\alpha_1}u_h=\pt_t u_h$ is piecewise constant in time, and it is convenient to treat it as a left-continuous function, viz., $\pt_t u_h=\delta_t^ju_h:=[u_h(\cdot, t_j)-u_h(\cdot, t_{j-1})]/(t_j-t_{j-1})$ is constant in time
on each time interval $(t_{j-1},t_j]$.
As before, one gets $R_h(\cdot, t_j)=0$ for $j\ge 1$, so
$R_h^I=[\LL u_0-f(\cdot, 0)](1-t/t_1)^+$ for $t>0$ --- but $R_h$ is no longer continuous in time. (To be precise,
$R_h-q_1(t)\, \pt_t u_h$ is continuous on $[0,T]$, assuming that $u_0$ is smooth; a modification for the case when $\LL u_0\notin L_2(\Omega)$ is discussed in \cite[Remark 2.7]{ KopAML22}.)
Nevertheless,
we can still employ $R_h-R_h^I=(D_t^{\bar\alpha} u_h-f)-(D_t^{\bar\alpha} u_h-f)^I$,
but one needs to be more careful when evaluating the component $(q_1 \pt_t u_h)^I$
of $(D_t^{\bar\alpha})^I$: on each $(t_{j-1},t_j]$ with $j>1$, one gets
$$
(q_1 \pt_t u_h)^I=q_1^I\,\delta_t^j u_h-q_1(t_{j-1})\,\left[\delta_t^{j} u_h-\delta_t^{j-1} u_h\right]
	\frac{t_j-t}{t_j-t_{j-1}}
$$
(to check this formula, observe that it is linear in time and equals $q_1(t_j)\,\delta_t^j u_h $ at $t_j$ and $q_1(t_{j-1})\,\delta_t^{j-1} u_h $ at $t_{j-1}$).
On $[0,t_1]$, i.e., when $j=1$, the situation is simpler as $u_h$ is continuous in time, so
$(q_1\, \pt_t u_h)^I=q_1^I\,\delta_t^1 u_h$, so one can still employ the above formula after setting $\delta_t^{0} u_h:=\delta_t^{1} u_h$.
Thus, even when $\alpha_1=1$,
 one can still compute $R_h$ by sampling, using parallel/vector evaluations, without a direct application of $\LL$ to $\{u^j_h\}$.

Finally, for completeness we include in Figure~\ref{fig_alg}  a description of the
adaptive algorithm of \cite{KopAML22}, to aid the reader's understanding of the numerical results that follow. This algorithm is motivated by \eqref{R0bound} and \eqref{R1bound}; it constructs a temporal mesh such that $\|R_h(\cdot,t)\|\le \TOL\cdot\RR_p(t)$ for $p=0,1$, with $Q:=1.1, \ \tau_{**}:=0$ and $\tau_*:=5t_1$ in~$\RR_1$. (Experiments with larger values of $Q$ and a discussion of implementation of the algorithm are given in \cite{FraKopInPrep}.)    Note that the computation of the mesh in the algorithm is one-dimensional in nature and is independent of the number of spatial dimensions in \eqref{problem}, since it is based on the scalar quantity $\|R_h(\cdot,t)\|$.

  \begin{figure}[t!]
 \begin{center}
 ~\hfill\noindent
     \parbox{0.7\textwidth}{%

     \noindent\hrulefill

         $u_h^0:=u_0$;\;\;$t_0:=0$;\;\;$t_1:=\min\{\tau_{*},\,T\}$;\;\;$m:=0$;

         {\tt while}\;\;$t_m<T$

         \ind $m:=m+1$;\;\;$flag :=0$;

         \ind {\tt while}\;\;$t_{m}-t_{m-1}>\tau_{**}$

         \ind\ind compute $u_h^m$ using \eqref{L1method}

         \ind\ind {\tt if}\;\;$\|R_h(\cdot,t)\|\le\TOL\cdot \RR_p(t)$ $\forall\,t\in(t_{m-1},t_m)$

         \ind\ind\ind {\tt if}\;\;$t_m=T$

         \ind\ind\ind\ind $M:=m$;\;\;{\tt break}

         \ind\ind\ind {\tt elseif}\;\;$t_m<T$

         \ind\ind\ind\ind  $\tilde u_h^m := u_h^m$;\;\;$\tilde t_m := t_m$;

         \ind\ind\ind\ind
         $t_m:=\min\{t_{m-1}+Q(t_m-t_{m-1}) ,\,T\}$;\;\;$flag: =1$;

         \ind\ind\ind {\tt end}

         \ind\ind {\tt else}

         \ind\ind\ind {\tt if}\;\;$flag =0$

         \ind\ind\ind\ind $t_m:=t_{m-1}+(t_m-t_{m-1})/Q$;

         \ind\ind\ind {\tt else}

         \ind\ind\ind\ind $u_h^m := \tilde u_h^m$;\;\;$t_m := \tilde t_m$;

         \ind\ind\ind\ind
         $t_{m+1}:=\min\{t_m+(t_m-t_{m-1}),\, T\}$;\;\;{\tt break}

         \ind\ind\ind {\tt end}

         \ind\ind {\tt end}

         \ind {\tt end}

         {\tt end}

         \noindent\hrulefill
 }\hfill~%
 %
 \end{center}
 \vspace{-0.0cm}
  \caption{\label{fig_alg}\it\small
  Adaptive algorithm}
 \end{figure}

\subsection{Numerical results with $\alpha_1<1$}

We start our numerical experiments with three \tcb{initial-value problems of the form~\eqref{IVP} to illustrate orders of convergence, since time discretisation is the main focus of our paper. A  subdiffusion  test problem of the form \eqref{problem} (i.e., containing spatial and temporal derivatives)} will then be considered.

As well as results computed  on our adaptive mesh, some of the figures \tcb{compare the adaptive mesh itself with} the $(M+1)$-point \emph{graded  mesh} $t_k := T(k/M)^r$ for $k=0,1,\dots,M$ that is often used in conjunction with the L1 scheme (see \cite{StyL1survey}). Here $r\ge 1$ is a user-chosen mesh grading parameter and it is known
\cite{KopMC19,SORG17} that when $\ell=1$ the choice $r=(2-\alpha)/\alpha$ yields the optimal mesh grading for the problem~\eqref{problem}; we make an analogous choice of $r$ in our experiments. \tcb{We shall see that the adaptive mesh constructed by our algorithm --- without using any information about the exact solution and without any guidance from the user --- is remarkably similar to the optimal graded mesh. Of course this holds great promise for the performance of the algorithm in problems where no a priori analysis of the exact solution (and therefore no optimal a priori mesh) is available.}

\tcb{To begin, we present three initial-value examples to demonstrate that an adaptive approach based on our a posteriori analysis works well in widely-differing regimes.}

\begin{example}\label{exa:1}
\emph{Consider \eqref{problem} without spatial derivatives, with $\LL:=1$, $T=1$, and $\ell=2$, and
\beq\label{testA}
\alpha_1=\alpha,\quad
\alpha_2={\textstyle\frac23}\alpha,\quad
q_1(t)={\textstyle\frac12}e^{-t/5}, \quad q_2(t)=1-q_1(t),\quad
u(0)=0,\quad
f(t)\equiv 1,
\eeq
where $\alpha\in(0,1)$.
\tcb{In this example one has $q_1(t)>0$ and $q_2(t)>0$ for all~$t$.}
The unknown exact solution is replaced by a reference solution (computed on a considerably finer mesh).
See Figures~\ref{fig_TestA_R0} and~\ref{fig_TestA_R1} for errors  in the computed solutions and the meshes generated.  }
\end{example}

  \begin{figure}[h!]
\begin{center}
\hspace*{-0.3cm}\includegraphics[height=0.25\textwidth]{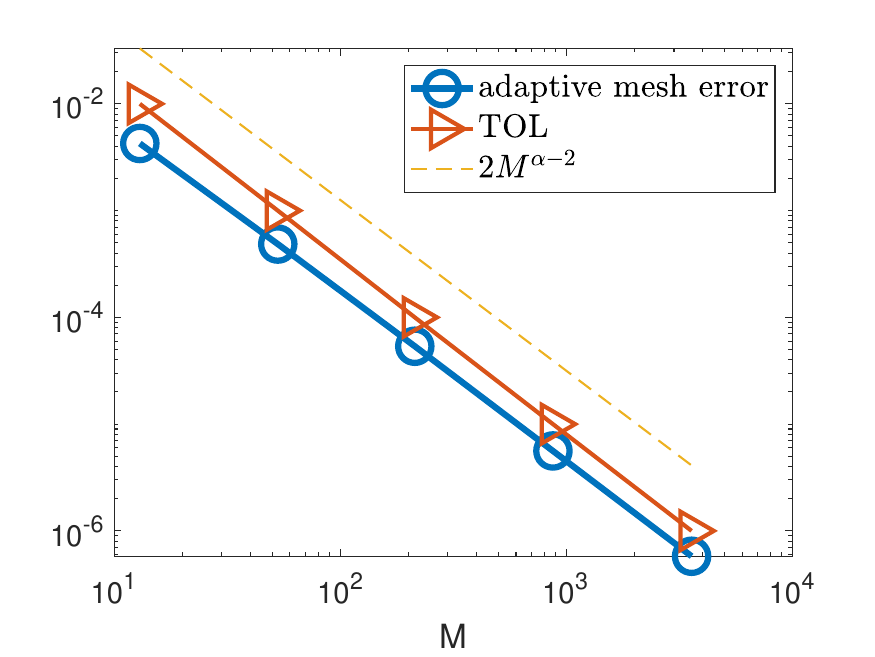}\hspace*{-0.0cm}%
\includegraphics[height=0.25\textwidth]{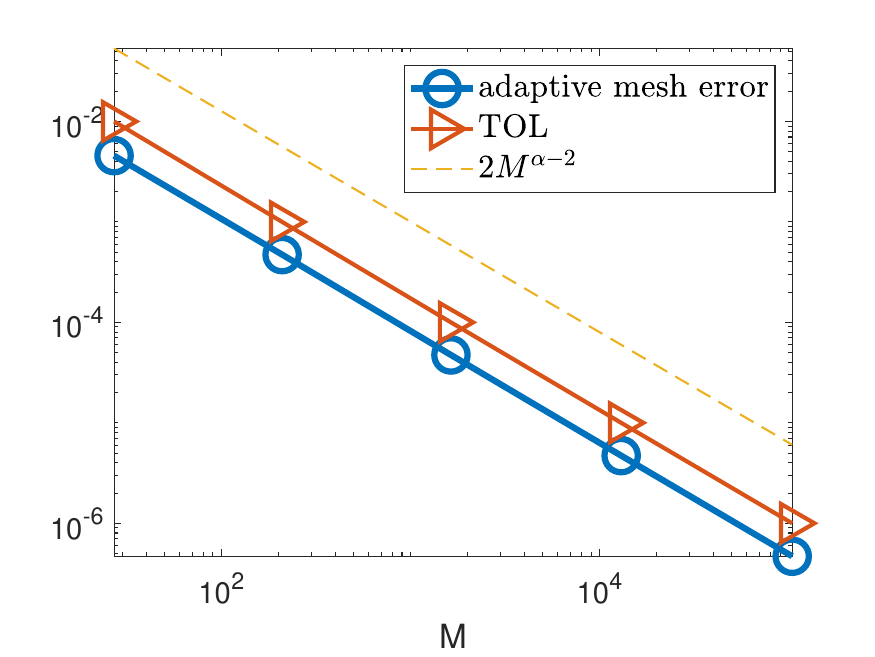}%
\hspace*{-0.0cm}\includegraphics[height=0.25\textwidth]{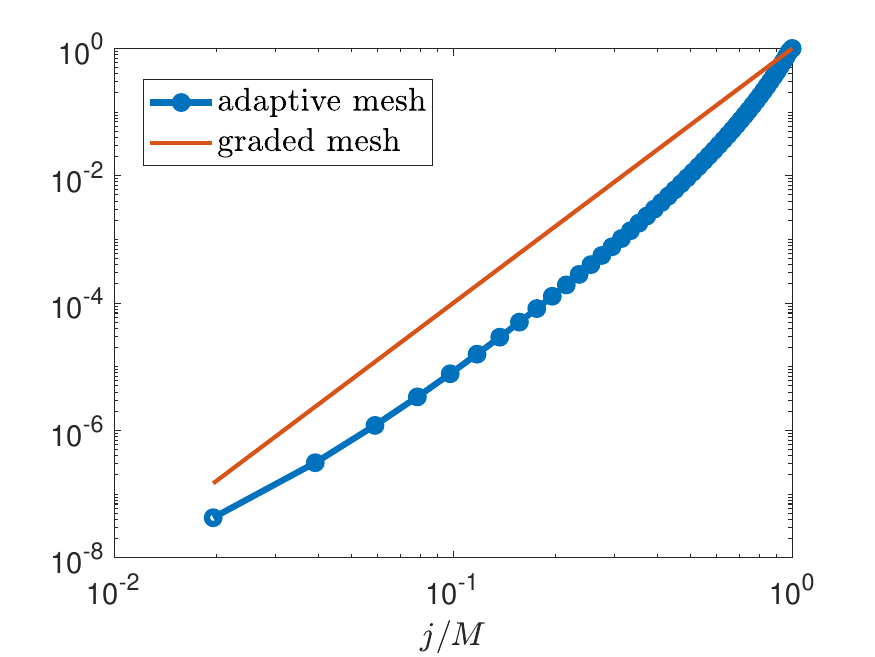}\hspace*{-0.45cm}
\end{center}
 \caption{\label{fig_TestA_R0}\it\small
Adaptive algorithm with $\RR_0(t)$ for Example~\ref{exa:1}: loglog graphs of 
$\max_{[0,T]}\vert e(t)\vert$ on the adaptive mesh and the corresponding $\TOL$, 
for $\alpha=0.4$ (left) and $\alpha=0.9$ (centre).
Right: loglog graphs of $\{t_j\}_{j=0}^M$ as a function of $j/M$ for
our adaptive mesh and the standard graded mesh with $r=(2-\alpha)/\alpha$,
 $\alpha=0.4$, $\TOL =10^{-3}$, $M=51$.}
 \end{figure}

  \begin{figure}[h!]
\begin{center}
\hspace*{-0.3cm}\includegraphics[height=0.25\textwidth]{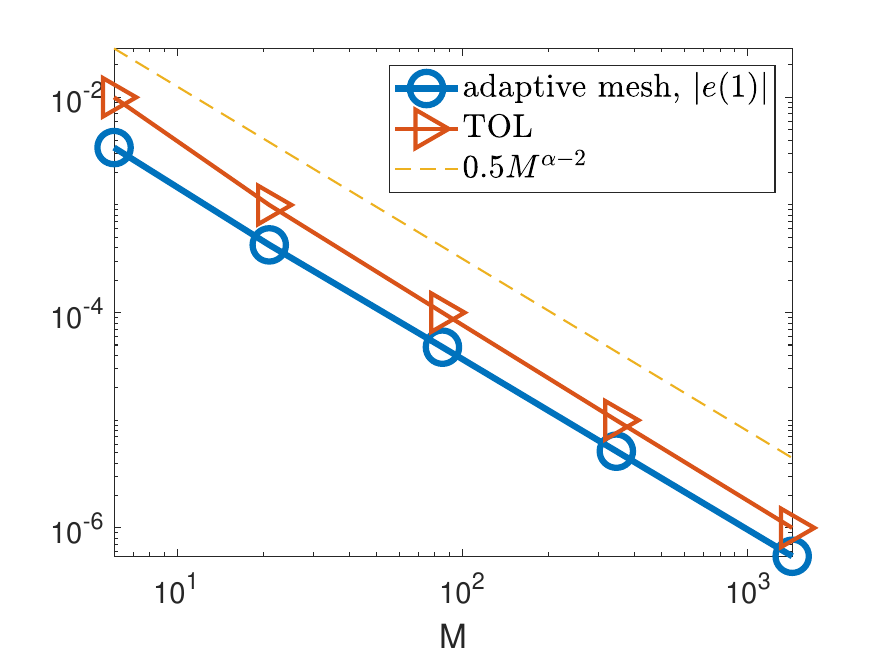}\hspace*{-0.0cm}%
\includegraphics[height=0.25\textwidth]{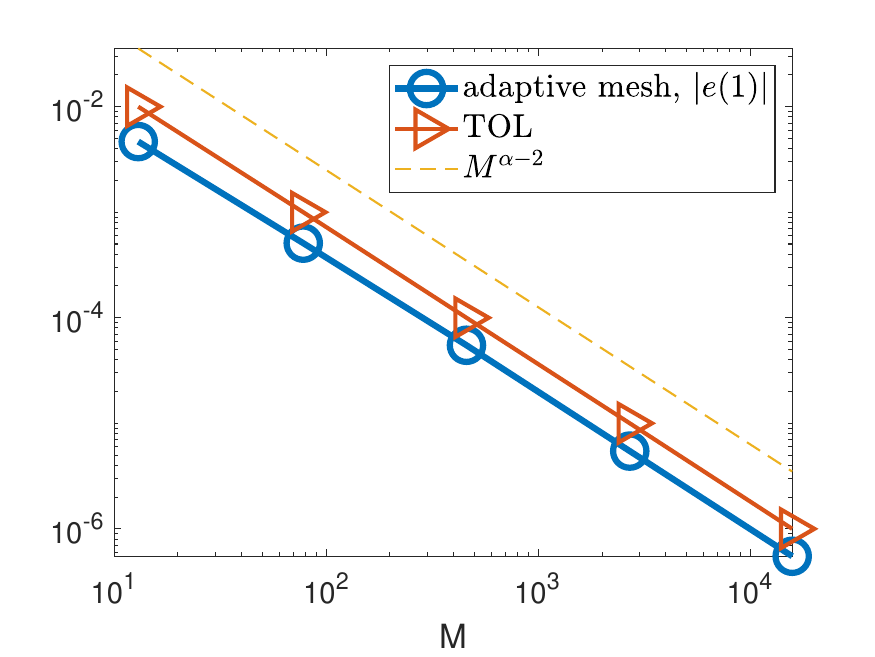}%
\hspace*{-0.0cm}\includegraphics[height=0.25\textwidth]{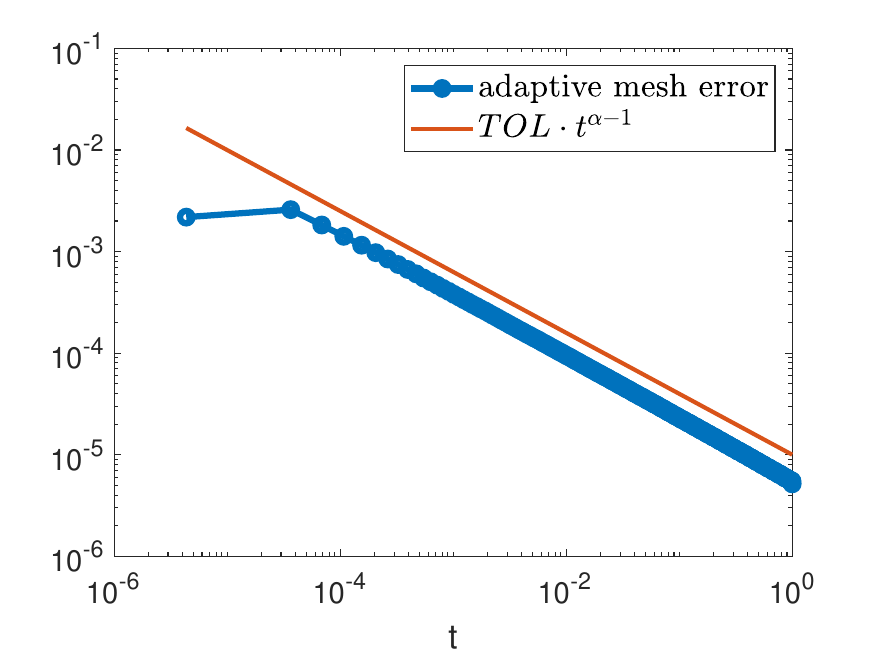}\hspace*{-0.45cm}
\end{center}
 \caption{\label{fig_TestA_R1}\it\small
 Adaptive algorithm with $\RR_1(t)$ for Example~\ref{exa:1}: 
$\vert e(1)\vert$ on the adaptive mesh and the corresponding $\TOL$,
for $\alpha=0.4$ (left) and $\alpha=0.7$ (centre).
Right:
log-log graph of the pointwise error $\vert e(t_j)\vert$ on the adaptive mesh and $\TOL\cdot t^{\alpha-1}$
for $\alpha=0.4$, $\TOL =10^{-5}$, $M=346$.}
 \end{figure}


\begin{example}\label{exa:2}
\emph{We modify Example~\ref{exa:1} by resetting
$$
q_1(t):=\cos^2(\pi t)\quad\mbox{for}\;\;t<{\textstyle\frac12},\qquad
q_1(t):=0\quad\mbox{for}\;\;t\ge{\textstyle\frac12},\qquad q_2(t):=1-q_1(t),
$$
while retaining $u(0)=0$ and $f(t)\equiv1$.
\tcb{Now the coefficient of the highest-order derivative vanishes for $t\ge 1/2$.}
Loglog graphs of reference solutions indicate that the solution to this problem has an initial singularity
of type $t^{\alpha_1}$ (compare the constant-coefficient analysis of Section~\ref{ssec_const}) and remains smooth away from $t=0$. See Figure~\ref{fig_TestB_R0} for errors  in the computed solutions and the mesh generated.  We also display (see rightmost figure) the meshes generated when $f(t)=\cos(5t^2)$ to show that the algorithm continues to perform well when $f$ changes rapidly.  }
\end{example}

  \begin{figure}[h!]
\begin{center}
\hspace*{-0.3cm}\includegraphics[height=0.25\textwidth]{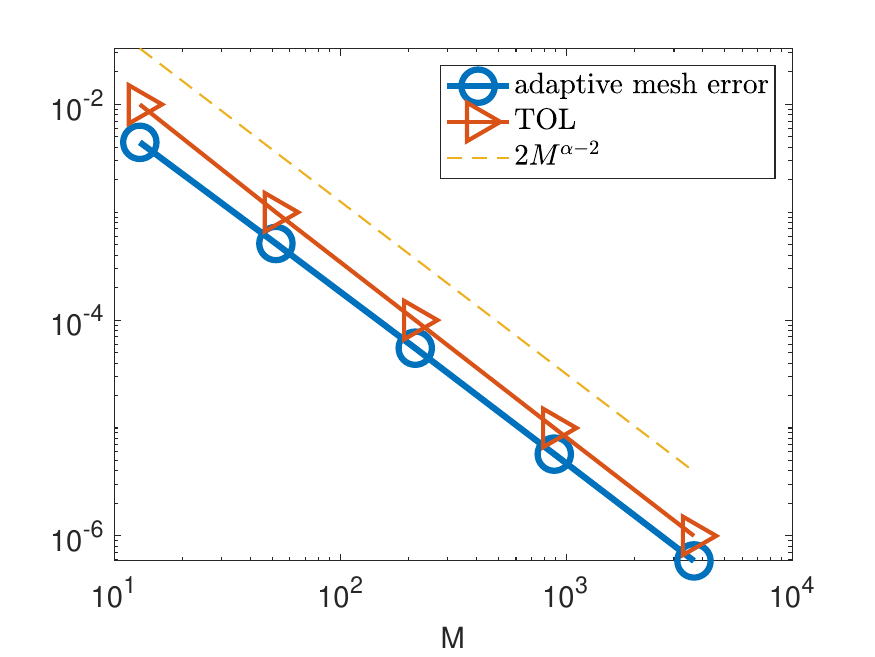}\hspace*{-0.0cm}%
\includegraphics[height=0.25\textwidth]{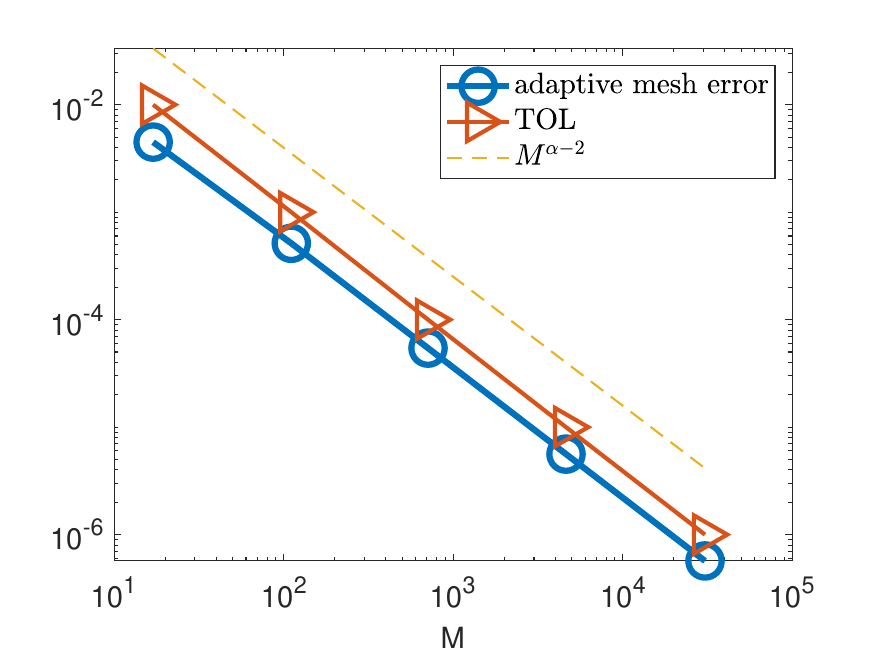}%
\hspace*{-0.0cm}\includegraphics[height=0.25\textwidth]{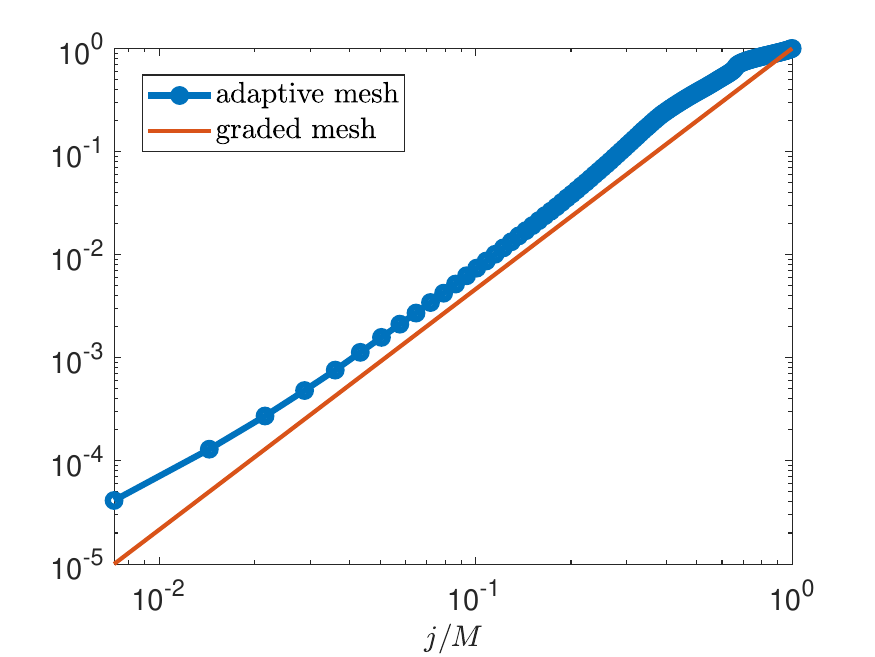}\hspace*{-0.45cm}
\end{center}
 \caption{\label{fig_TestB_R0}\it\small
Adaptive algorithm with $\RR_0(t)$ for Example~\ref{exa:2}: loglog graphs of 
$\max_{[0,T]}\vert e(t)\vert$ on the adaptive mesh and the corresponding $\TOL$ 
for $\alpha=0.4$ (left) and $\alpha=0.8$ (centre).
Right: Change $f$ to $f(t)=\cos(5t^2)$; loglog graphs of $\{t_j\}_{j=0}^M$ as a function of $j/M$ for
our adaptive mesh and the standard graded mesh with $r=(2-\alpha)/\alpha$
for $\alpha=0.6$, $\TOL =10^{-3}$, $M=139$.}
 \end{figure}

\begin{example}\label{exa:3}
\emph{We modify Example~\ref{exa:1} by resetting
$$
q_1(t):=0\quad\mbox{for}\;\;t<{\textstyle\frac12},\qquad
q_1(t):=\cos^2(\pi t)\quad\mbox{for}\;\;t\ge{\textstyle\frac12},\qquad q_2(t):=1-q_1(t).
$$
\tcb{Here the situation is opposite to that of Example~\ref{exa:2}:  the coefficient of the highest-order derivative vanishes for $t< 1/2$.}
Loglog graphs of reference solutions indicate that the solution to this problem has an initial singularity
of type $t^{\alpha_2}$ (one could show this analytically by an extension of Remark~\ref{rem:w'(t)}) and remains smooth away from $t=0$.
See Figure~\ref{fig_TestC_R0} for errors  in the computed solutions and the mesh generated. }
\end{example}

  \begin{figure}[h!]
\begin{center}
\hspace*{-0.3cm}\includegraphics[height=0.25\textwidth]{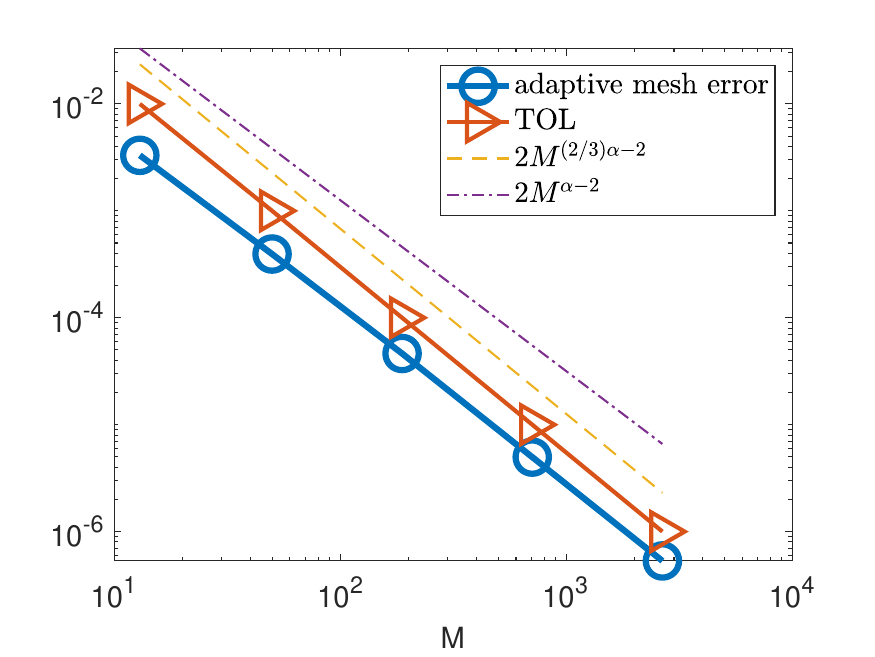}\hspace*{-0.0cm}%
\includegraphics[height=0.25\textwidth]{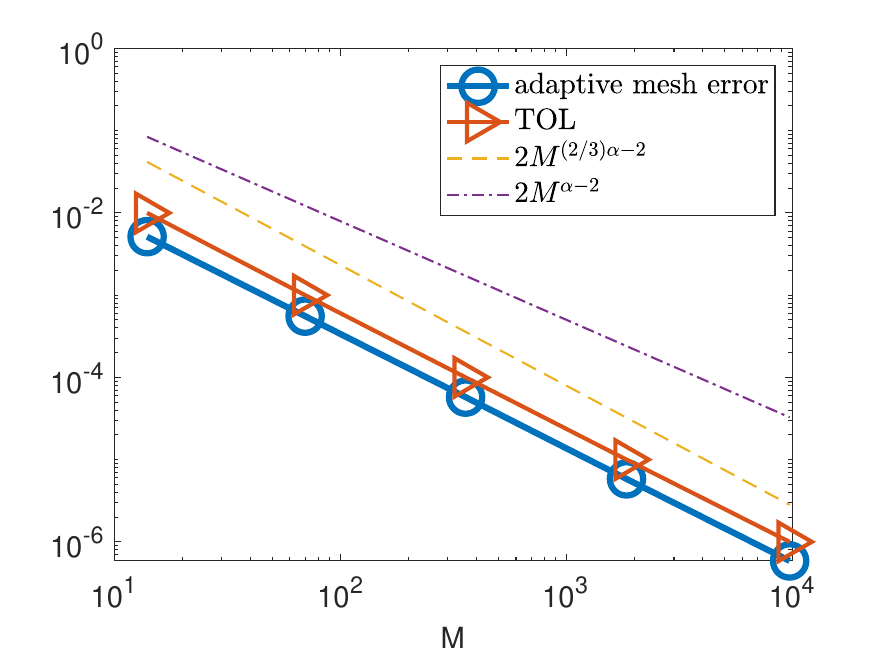}%
\hspace*{-0.0cm}\includegraphics[height=0.25\textwidth]{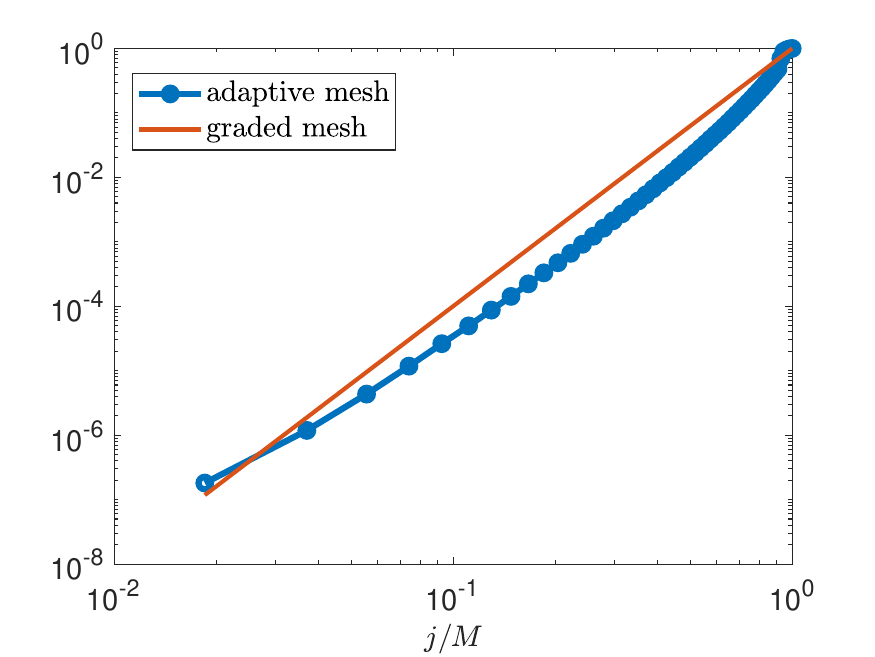}\hspace*{-0.45cm}
\end{center}
 \caption{\label{fig_TestC_R0}\it\small
Adaptive algorithm with $\RR_0(t)$ for Example~\ref{exa:3}: loglog graphs of 
$\max_{[0,T]}\vert e(t)\vert $ on the adaptive mesh and the corresponding $\TOL$ 
for $\alpha=0.4$ (left) and $\alpha=0.8$ (centre).
Right: 
loglog graphs of $\{t_j\}_{j=0}^M$ as a function of $j/M$ for
our adaptive mesh and the standard graded mesh with $r=(2-\alpha_2)/\alpha_2$,
 $\alpha=\alpha_1=0.6$, $\TOL =10^{-3}$, $M=54$.}
 \end{figure}

\begin{example}\label{exa:4}
\emph{Now we consider the subdiffusion analogue \eqref{problem} of \eqref{testA}: retain the values of $\alpha_1, \alpha_2, q_1, q_2$ and set
\[
u_0(x)=\sin(x^2/\pi),\qquad \Omega =(0,\pi),\qquad \lambda=1,\qquad \LL=-{\textstyle\frac{d^2}{dx^2}},
\qquad f\equiv 1.
\]
\tcb{Note that the initial data $u_0$ has only limited compatibility with the other data at the corner $(\pi,0)$ of the space-time domain. Nevertheless the algorithm performs satisfactorily. (Related examples where either the exact solution is known, or the initial condition is piecewise linear, were  tested in~\cite{KopAML22}.)}
See Figure~\ref{fig_parabolic} for errors in the computed solutions. }
\end{example}

   \begin{figure}[h!]
\begin{center}
\hspace*{-0.3cm}\includegraphics[height=0.25\textwidth]{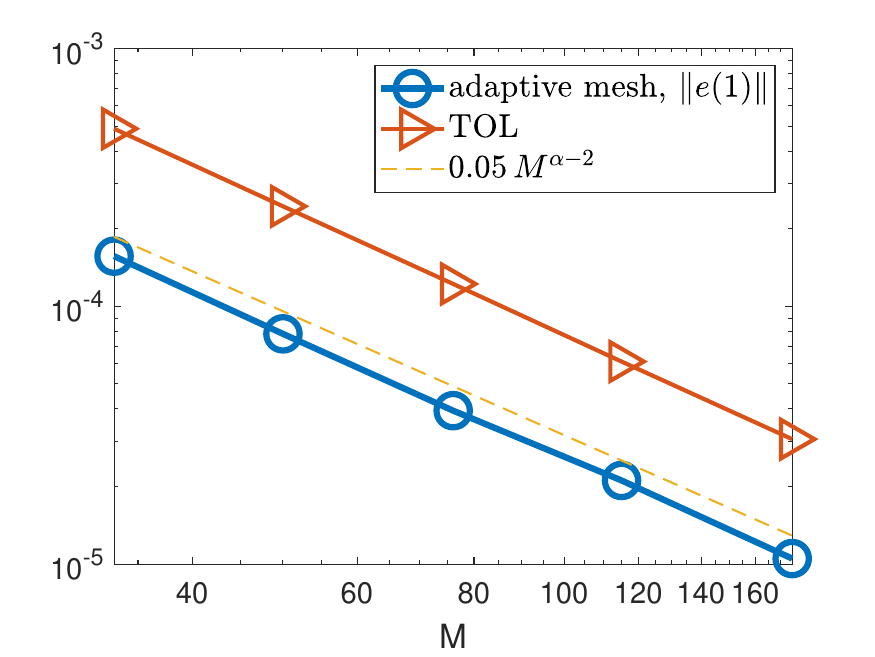}\hspace*{-0.0cm}%
\includegraphics[height=0.25\textwidth]{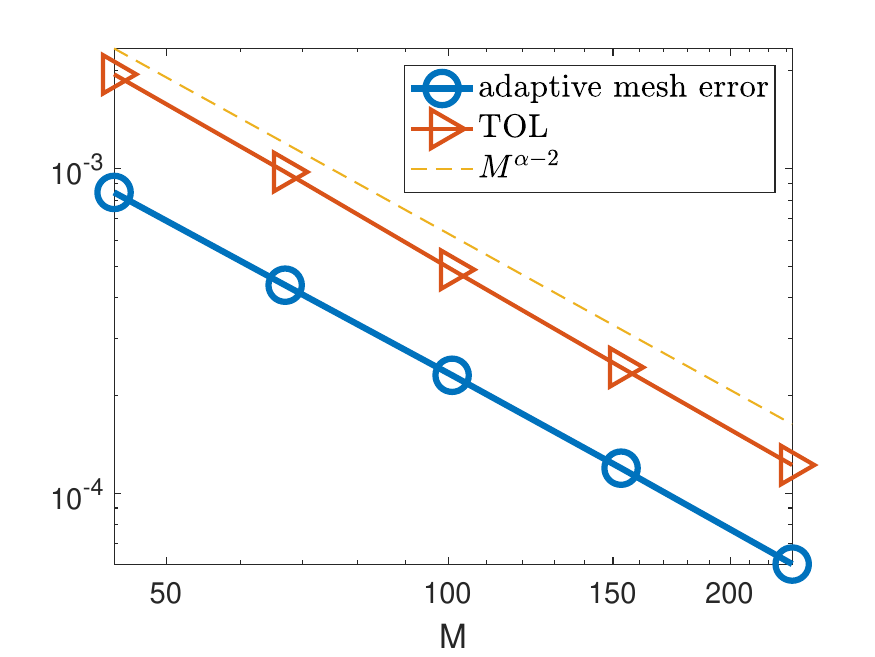}%
\hspace*{-0.0cm}\includegraphics[height=0.25\textwidth]{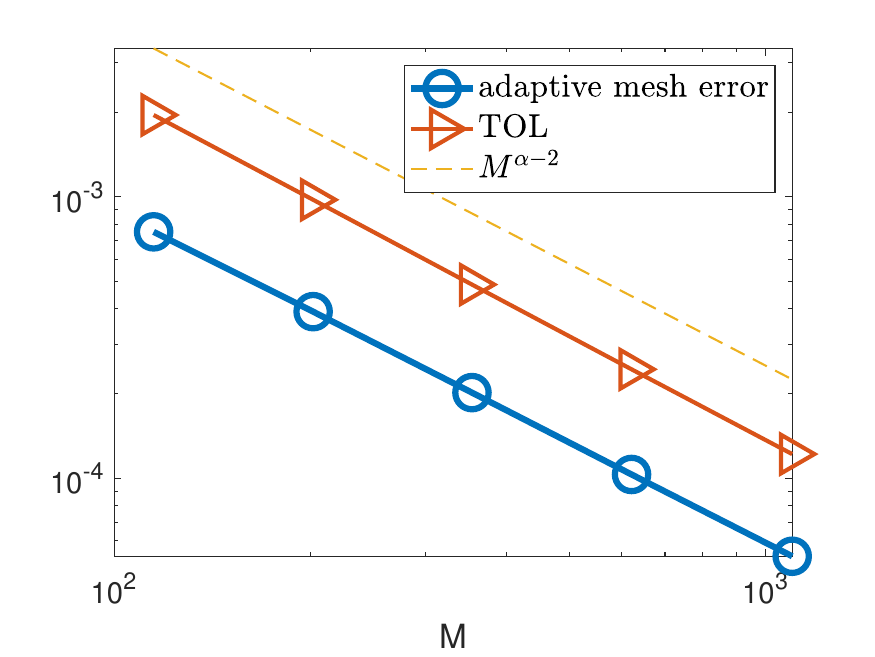}\hspace*{-0.45cm}
\end{center}
 \caption{\label{fig_parabolic}\it\small
Example~\ref{exa:4} adaptive algorithm results: (left) for $\RR_1(t)$ with $\alpha=0.4$, showing $\|e(1)\|$ and  $\TOL$;
for $\RR_0(t)$ with  $\alpha=0.4$ (centre) and $\alpha=0.8$ (right),  $\max_{t_j\in(0,T]}\|e(t_j)\|$ on the adaptive mesh
 and the corresponding $\TOL$.}
 \end{figure}

 \subsection{Numerical results with $\alpha_1=1$}

\begin{example}\label{exa:5}
\emph{Consider the IVP \eqref{IVP} with $\alpha_1=1, \lambda=1$ and
\begin{align*}
&q_1(t):=c_1e^{-5t}\cos^2(\pi t)\;\mbox{for}\;t<{\textstyle\frac12},\quad
	q_1(t):=0\;\mbox{for}\;t\ge{\textstyle\frac12},\quad q_2(t):=1-q_1(t), \\
&f(t):=1+{\textstyle\frac12}{\rm erf}(20(1-t)).
\end{align*}
See Figure~\ref{fig_ODE_alpha1} for results for $c_1=1$ and Figure~\ref{fig_ODE_alpha1_12} for those for $c_1=\frac12$.
When $c_1=1$, the solution has no initial singularity and we used the exponential barrier function $\EE(t):=1-\exp(-10t)$ since it gives better results in this case. For $c_1=\frac12$ one has $q_2(0)>0$, so we employed $\EE_0$ and hence $\RR_0$ as in the earlier examples for $\alpha_1<1$. }

\emph{Note: when evaluating $\RR(t) :=\left( \sum_{i=1}^{\ell} q_i(t) D _t ^{\alpha_i} + \lambda\right) \EE(t)$ in~\eqref{res_bound},
$D_t^{\alpha_1}\EE=\EE'(t)$ is computed explicitly, while $D_t^{\alpha_2}\EE$ is computed using quadrature. }
\end{example}

   \begin{figure}[h!]
\begin{center}
\hspace*{-0.3cm}\includegraphics[height=0.25\textwidth]{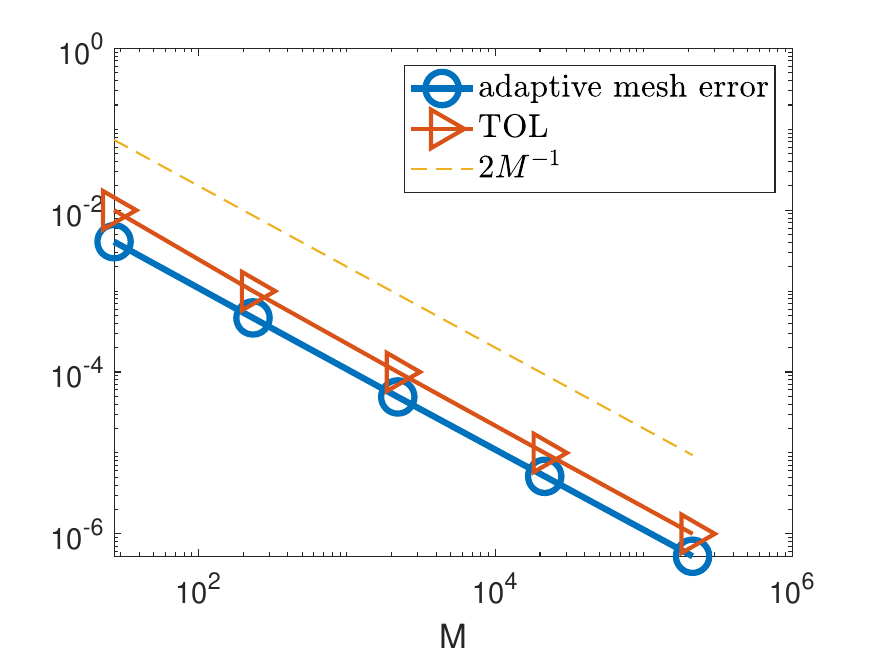}\hspace*{-0.0cm}%
\includegraphics[height=0.25\textwidth]{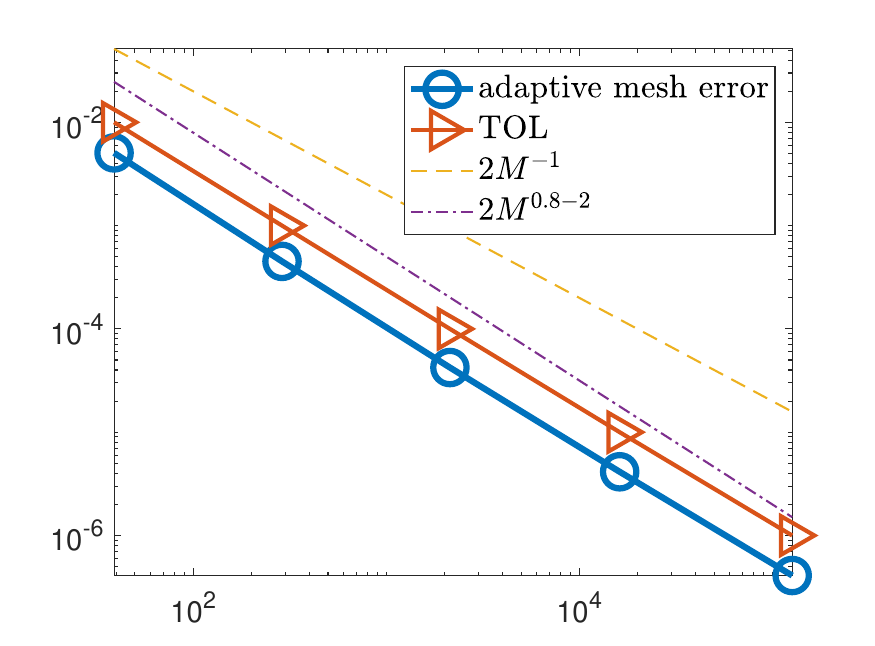}%
\hspace*{-0.0cm}\includegraphics[height=0.25\textwidth]{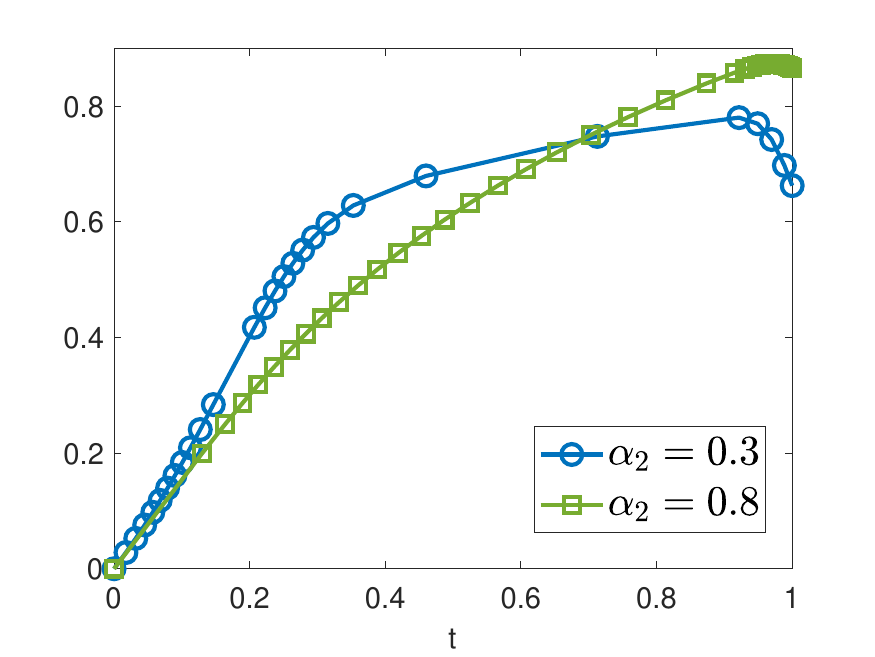}\hspace*{-0.45cm}
\end{center}
 \caption{\label{fig_ODE_alpha1}\it\small
Adaptive algorithm with $\RR(t)$ generated by $\EE=1-\exp(-10t)$ for Example~\ref{exa:5} with $\alpha_1=1$ and $c_1=1$: loglog graphs of 
$\max_{[0,T]}\vert e(t)\vert$ on the adaptive mesh and the corresponding $\TOL$ 
for $\alpha_2=0.3$ (left) and $\alpha_2=0.8$ (centre).
Right: computed solutions for this test problem obtained using $\TOL =10^{-2}$.}
 \end{figure}

    \begin{figure}[h!]
\begin{center}
\hspace*{-0.3cm}\includegraphics[height=0.25\textwidth]{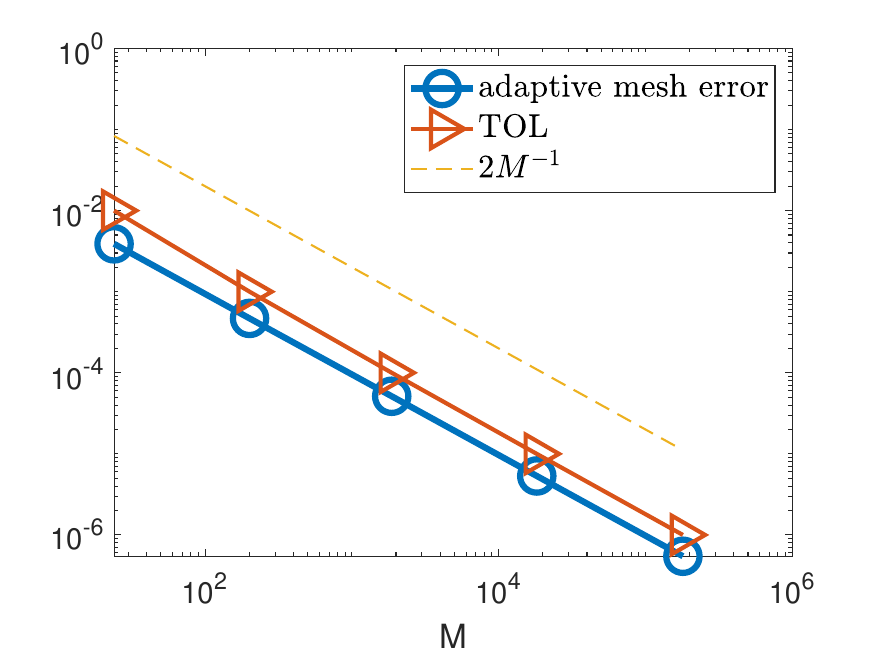}\hspace*{-0.0cm}%
\includegraphics[height=0.25\textwidth]{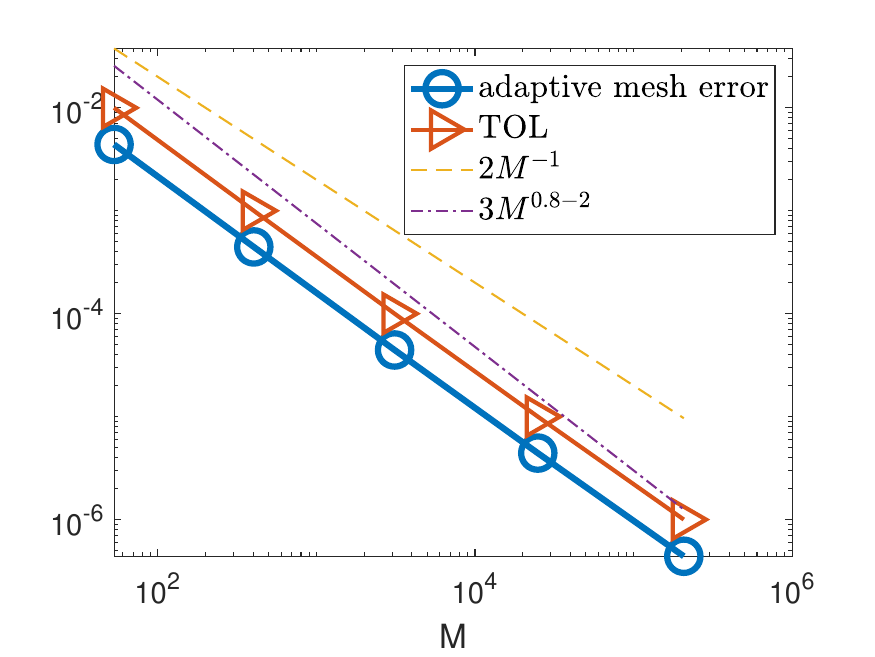}%
\hspace*{-0.0cm}\includegraphics[height=0.25\textwidth]{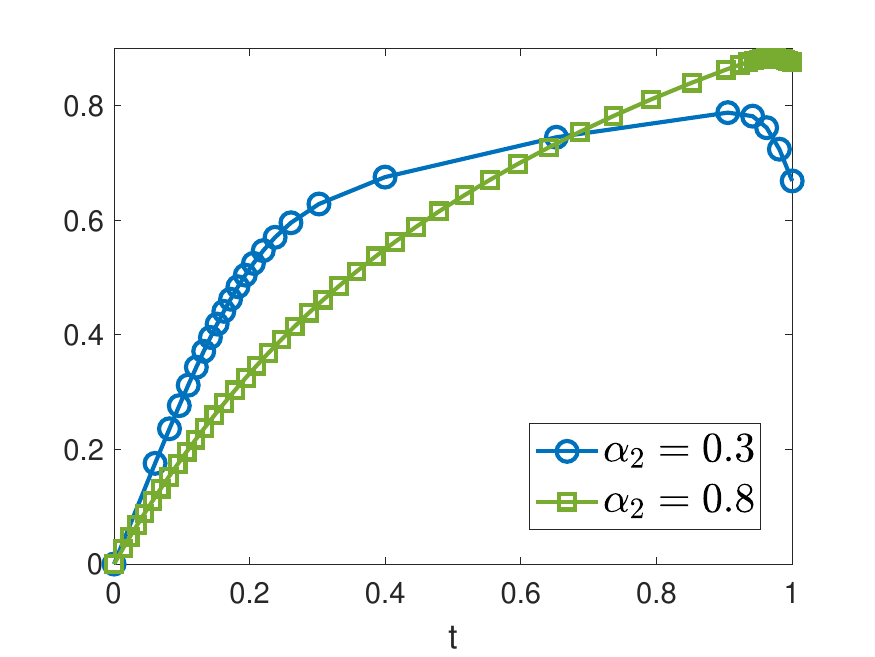}\hspace*{-0.45cm}
\end{center}
 \caption{\label{fig_ODE_alpha1_12}\it\small
Adaptive algorithm with $\RR_0(t)$  for Example~\ref{exa:5} with $\alpha_1=1$ and $c_1=\frac12$: loglog graphs of 
$\max_{[0,T]}\vert e(t)\vert$ on the adaptive mesh and the corresponding $\TOL$ 
for $\alpha_2=0.3$ (left) and $\alpha_2=0.8$ (centre).
Right: computed solutions for this test problem obtained using $\TOL =10^{-2}$.}
 \end{figure}

In the next example we return to our subdiffusion problem \eqref{problem}.

\begin{example}\label{exa:6}
\emph{Take $q_1$, $q_2$ and $f$ as in Example~\ref{exa:5}, with $c_1=\frac12$, while $\LL$, $u_0$, $\Omega$ and $\lambda=1$ are taken from Example~\ref{exa:4}.
Now we choose the temporal grid a priori to be uniform. Once the computed solution is obtained, we compute the residual
$\|R_h(\cdot, t)\|$ on a finer mesh, with 15 equidistant additional points between any consecutive time layers. }

\emph{Assuming that there exists a solution $\EE$ of $\left( \sum_{i=1}^{\ell} q_i(t)\, D _t ^{\alpha_i} +\lambda \right)\EE=\|R_h(\cdot, t)\|$, inequality \eqref{L2_error} gives an upper bound for the error, viz., $\|(u_h-u)(\cdot,t)\|\le \EE$. In practice, one finds a numerical approximation $\EE_h$ of $\EE$ on the above fine grid. }

\emph{It is important to note that the computed solution $u_h$ is a numerical approximation of the fractional  subdiffusion  problem with spatial derivatives, while the computation of $\EE_h$, although the latter is computed on a much finer temporal grid, is inexpensive, as $\EE(t)$ is a solution of an initial-value problem without spatial derivatives. }

\emph{See Figure~\ref{fig_par_uniform_grid} for results. }
\end{example}

    \begin{figure}[h!]
\begin{center}
\hspace*{-0.3cm}\includegraphics[height=0.25\textwidth]{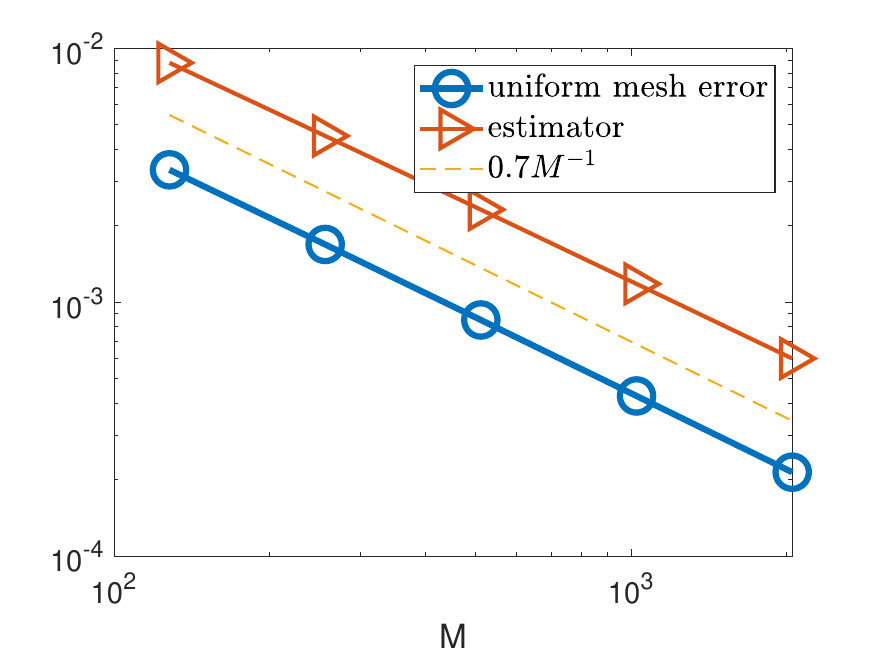}\hspace*{-0.0cm}%
\includegraphics[height=0.25\textwidth]{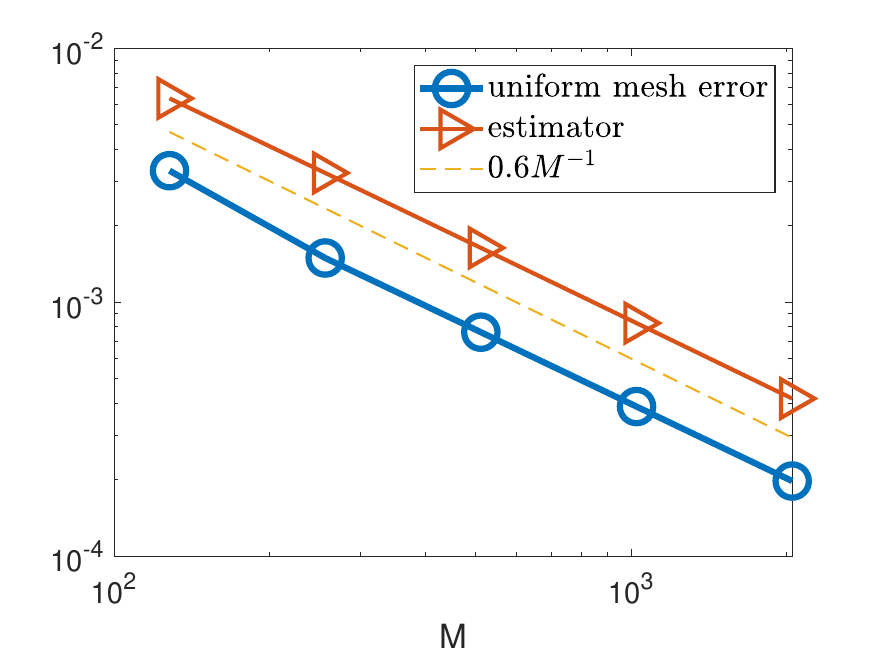}%
\hspace*{-0.0cm}\includegraphics[height=0.25\textwidth]{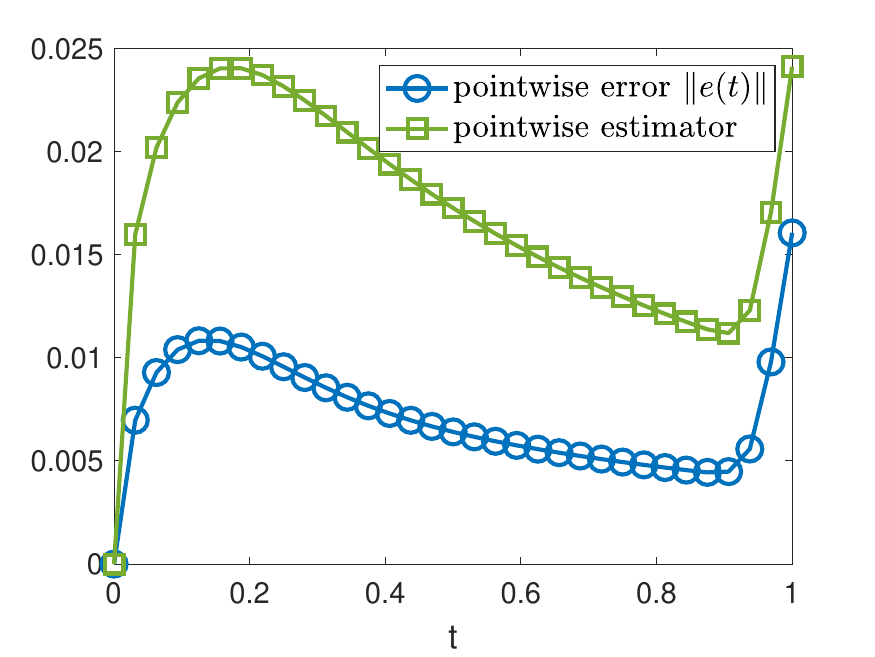}\hspace*{-0.45cm}
\end{center}
 \caption{\label{fig_par_uniform_grid}\it\small
A posteriori error estimation on the uniform temporal mesh for Example~\ref{exa:6} with $\alpha_1=1$
 and $c_1=\frac12$: loglog graphs of 
$\max_{[0,T]}\|e(t_j)\|$ and the corresponding estimator $\max_{t_j\in[0,T]}\EE_h(t_j)$ 
for $\alpha_2=0.3$ (left) and $\alpha_2=0.8$ (centre).
Right: pointwise-in-time error $\|e(t_j)\|$ and pointwise estimator $\EE_h(t_j)$ for $\alpha_2=0.8$, $M=32$.}
 \end{figure}

The numerical results in this section demonstrate that,  for many different types of data,  our algorithm based on the L1 scheme automatically adapts the given initial mesh to compute accurate numerical solutions. It gives excellent results for problems whose solutions have a weak singularity at $t=0$, without requiring the user to choose a suitable mesh --- while if the mesh is prescribed a priori, it can estimate the error in the  solution computed on this mesh (see Figure~\ref{fig_par_uniform_grid}).  It is equally good in cases where this weak singularity is absent.

\appendix
\section{A variant of Lemma~\ref{lem:v0w0}}\label{sec:app}
In this appendix we shall  prove a result (Lemma~\ref{lem:IVPhomog}) that complements Lemma~\ref{lem:v0w0}. This is done by using an explicit complex contour integral formula to derive a positivity property of the multinomial Mittag-Leffler function (Lemma~\ref{lem:beta>1}) that appears to be new.

Our argument starts with the following elementary result.

\begin{lemma}\label{lem:S(s)}
Let $m$ and $n$ be nonnegative integers with $m<n$. Set $S(s) = \sum_{j=0}^n k_js^{\gamma_j}$ for $s\in [0,\infty)$, where $0 = \gamma_0 < \gamma_1 <\dots < \gamma_n  \le 1$ and $k_j>0$ for $0\le j\le m$, \ $k_j<0$ for $m<j\le n$. Then the equation $S(s) =0$ has a unique solution $s_0\in (0,\infty)$, with $S(s)>0$ for $0\le s <s_0$ and $S(s)<0$ for $s_0 <s < \infty$.
\end{lemma}
\begin{proof}
If $t\in (0,\infty)$, then
\begin{align*}
S'(t) &=  \sum_{j=1}^n k_j \gamma_j t^{\gamma_j-1}
	=  \sum_{j=1}^m k_j \gamma_j t^{\gamma_j-1} -  \sum_{j=m+1}^n \vert k_j\vert  \gamma_j t^{\gamma_j-1} \\
	&\le \gamma_{m+1}t^{-1}\sum_{j=1}^m k_j  t^{\gamma_j} - \gamma_{m+1}t^{-1} \sum_{j=m+1}^n \vert k_j\vert t^{\gamma_j} \\
	&= \gamma_{m+1}t^{-1}\sum_{j=1}^n k_j  t^{\gamma_j} \\	
	&< \gamma_{m+1}t^{-1}\sum_{j=0}^n k_j  t^{\gamma_j} = \gamma_{m+1}t^{-1}S(t). 	
\end{align*}
Hence
\begin{equation}\label{S't}
S(t)\le 0 \ \text{ implies }S'(t)<0.
\end{equation}

One has $S(0)=k_0>0$. As $s\to\infty$, the term $k_ns^{\gamma_n}$ in $S(s)$ will dominate;  it follows that $S(s)<0$ for all sufficiently large~$s$. Hence $S(s)=0$ has at least one solution $s_0$ in $(0,\infty)$.

From \eqref{S't} it follows that $S'(s_0)<0$, so $S(s)<0$ on some interval $(s_0, s_0+\delta)$. Now \eqref{S't} ensures that $S$ can never reach a minimum on $(s_0, \infty)$, which implies that $S(s)<0$ for $s\in (s_0,\infty)$. Thus $s_0$ is the unique solution of $S(s)=0$.
\end{proof}

We now prove  a new positivity property of the multinomial Mittag-Leffler function that \tcb{is related to} Lemma~\ref{lem:beta<1}; this proof is in the spirit of classical analyses of Mittag-Leffler functions.  The argument used is partly based on \cite[pp.215--216]{PS11}, where a similar result was obtained for the simpler case of the  two-parameter Mittag-Leffler function $E_{\alpha,\beta}(t)$.

\begin{lemma}\label{lem:beta>1}
Assume  that $1 < \beta < 1+\alpha_1$ and $\lambda>0$. Then
\[
\tilde E(t):= E_{(\alpha_1,  \alpha_1-\alpha_\ell, \dots, \alpha_1-\alpha_3, \alpha_1-\alpha_2),\beta}
	(-\lambda t^\alpha_1, -q_\ell t^{ \alpha_1-\alpha_\ell}, \dots, -q_3t^{ \alpha_1-\alpha_3}, -q_2t^{ \alpha_1-\alpha_2})
	> 0
\]
for all $t> 0$.
\end{lemma}
\begin{proof}
For each $t>0$, by Remark~\ref{rem:Esymm} and \cite[eq.(47)]{LG99} we have
\begin{align*}
\tilde E(t) &=  E_{( \alpha_1-\alpha_2, \alpha_1-\alpha_3,\dots, \alpha_1-\alpha_\ell,\alpha_1),\beta}
	(-q_2t^{ \alpha_1-\alpha_2},-q_3t^{ \alpha_1-\alpha_3},\dots,-q_\ell t^{ \alpha_1-\alpha_\ell},-\lambda t^{\alpha_1}) \notag\\
	&= \frac{t^{1-\beta}}{2\pi i} \int_{\gamma(R_m, -\pi, \pi)}
		\frac{e^{\zeta t} \zeta^{\alpha_1-\beta}}{\zeta^{\alpha_1} + \sum_{j=2}^\ell q_j\zeta^{\alpha_j} +\lambda}\,d\zeta,
\end{align*}
where $r := \max\left\{ 1, \left(\lambda+\sum q_j\right)^{1/(\alpha_1-\alpha_2)} \right\}$, and ${\gamma(R, \theta_1,\theta_2)}$ (for $R\ge 0$ and $-\pi\le\theta_1\le\theta_2\le\pi$) denotes the complex-plane Hankel contour that comprises the ray $\arg \zeta = \theta_1$ with $\vert\zeta\vert\ge R$, the arc $\vert\zeta\vert=R$ with $\theta_1\le \arg\zeta\le\theta_2$, and the ray $\arg \zeta = \theta_2$ with $\vert\zeta\vert\ge R$, and the contour is traversed in the direction of increasing $\arg\zeta$.

The substitution $w= \zeta^{\alpha_1}$ gives
\begin{equation}\label{E1}
\tilde E(t) =  \frac{t^{1-\beta}}{2\alpha_1\pi i} \int_{\gamma(r^{\alpha_1}, -\alpha_1\pi, \alpha_1\pi)}
		\frac{w^{(1-\beta)/\alpha_1} \exp\left( t w^{1/\alpha_1} \right)}
			{w+\sum_{j=2}^\ell q_jw^{\alpha_j/\alpha_1} +\lambda}\,dw
\end{equation}
Observe that if $w\in\gamma(r^{\alpha_1}, -\alpha_1\pi, \alpha_1\pi)$, then $\vert\arg w\vert \le \alpha_1\pi$ independently of $r$; hence
\[
\left\vert\arg\left(w+\sum_{j=2}^{\ell} q_jw^{\alpha_j/\alpha_1}\right)\right\vert \le \max\{ \alpha_1\pi, \alpha_2\pi,\dots, \alpha_\ell\pi \}
	=\alpha_1\pi < \pi,
\]
so (recall that $\lambda>0$) the denominator of the integrand will not vanish if we change the value of~$r$ in the contour, and consequently the value of the integral will not change (by Cauchy's integral theorem). Furthermore, we can permit $r\to 0$ because $\beta<1+\alpha_1$ ensures that the integral remains finite. Thus we can replace the contour $\gamma(r^{\alpha_1}, -\alpha_1\pi, \alpha_1\pi)$ in~\eqref{E1} by $\gamma(0, -\alpha_1\pi, \alpha_1\pi)$.

Next, set $w = se^{\pm i\alpha_1\pi}$ along the ray $\arg w = \pm\alpha_1\pi$ (choose same sign). This yields
\begin{align}
\tilde E(t) &= \frac{t^{1-\beta}}{2\alpha_1\pi i}  \left[  \int_{\infty}^0
	\frac{s^{(1-\beta)/\alpha_1}e^{-i(1-\beta)\pi} \exp\left( t s^{1/\alpha_1}e^{-i\pi} \right)e^{- i\alpha_1\pi}}
			{se^{-i\alpha_1\pi}+\sum_{j=2}^\ell q_js^{\alpha_j/\alpha_1}e^{-i\alpha_j\pi} +\lambda}  \right. \notag\\
	&\hspace{2cm}\left. + \int_0^\infty
		\frac{s^{(1-\beta)/\alpha_1}e^{i(1-\beta)\pi} \exp\left( t s^{1/\alpha_1}e^{i\pi} \right)e^{ i\alpha_1\pi}}
			{se^{i\alpha_1\pi}+\sum_{j=2}^\ell q_js^{\alpha_j/\alpha_1}e^{i\alpha_j\pi} +\lambda}  \right]\, ds \notag\\
	&=  \frac{t^{1-\beta}}{2\alpha_1\pi i} \int_0^\infty s^{(1-\beta)/\alpha_1} \exp\left(- t s^{1/\alpha_1}\right)
		\left[  \frac{e^{i\beta\pi}}{s+\xi} -  \frac{e^{-i\beta\pi}}{s+\bar\xi}  \right]\,ds,  \label{E2}
\end{align}
where $\xi:= \sum_{j=2}^\ell q_js^{\alpha_j/\alpha_1}e^{i(\alpha_1-\alpha_j)\pi} +\lambda e^{i\alpha_1\pi}$ and $\bar\xi$ is its complex conjugate. Now
\begin{align*}
\frac{e^{i\beta\pi}}{s+\xi} -  \frac{e^{-i\beta\pi}}{s+\bar\xi}
	&= \frac{s(e^{i\beta\pi}-e^{-i\beta\pi})+\bar\xi e^{i\beta\pi}-\xi e^{-i\beta\pi}}{s^2+\xi^2}
	= \frac{2is \sin \beta\pi + 2i \Im(\bar\xi e^{i\beta\pi})}{s^2+\xi^2}
	= \frac{2iv(s)}{s^2+\xi^2}\,,
\end{align*}
where $v(s) := s \sin \beta\pi + \sum_{j=2}^\ell q_j s^{\alpha_j/\alpha_1} \sin(\beta-\alpha_1+\alpha_j)\pi + \lambda\sin(\beta-\alpha_1)\pi$. Hence \eqref{E2} becomes
\begin{equation}\label{E3}
\tilde E(t) = \frac{t^{1-\beta}}{\alpha_1\pi } I(t), \quad\text{where }
	I(t):= \int_0^\infty  s^{(1-\beta)/\alpha_1} \exp\left(- t s^{1/\alpha_1}\right)\frac{v(s)}{s^2+\xi^2}\, ds.
\end{equation}

Note that $v(s)$ has exactly the same structure as $S(s)$ in Lemma~\ref{lem:S(s)}, since $0 < \alpha_j/\alpha_1 <1$, $1 < \beta < 1+\alpha_1$ and $\lambda>0$. Thus there exists $s_0>0$ such that $v(s)>0$ for $0<s<s_0$ and $v(s)<0$ for $s>s_0$. From Definition~\ref{def:multiML} we get $\tilde E(0) = 1/\Gamma(\beta)>0$. By continuity we can choose $t_0>0$ such that $\tilde E(t)>0$ on~$(0,t_0]$, which implies $I(t_0)>0$. That is, recalling the properties of~$s_0$,
\begin{equation}\label{E4}
 \int_0^{s_0}  s^{(1-\beta)/\alpha_1} \exp\left(- t_0 s^{1/\alpha_1}\right)\frac{v(s)}{s^2+\xi^2}\, ds
 	>  \int_{s_0}^\infty  s^{(1-\beta)/\alpha_1} \exp\left(- t_0 s^{1/\alpha_1}\right)\frac{|v(s)|}{s^2+\xi^2}\, ds.
\end{equation}
Then for any $t>t_0$, using \eqref{E4} we get
\begin{align*}
&\hspace{-2cm} \int_0^{s_0}  s^{(1-\beta)/\alpha_1} \exp\left(- t s^{1/\alpha_1}\right)\frac{v(s)}{s^2+\xi^2}\, ds \\
 	&\ge \exp\left(-(t-t_0) s_0^{1/\alpha_1}\right)
		\int_0^{s_0}  s^{(1-\beta)/\alpha_1} \exp\left(- t_0 s^{1/\alpha_1}\right)\frac{v(s)}{s^2+\xi^2}\, ds    \\
	&>  \exp\left(-(t-t_0) s_0^{1/\alpha_1}\right)
		\int_{s_0}^\infty  s^{(1-\beta)/\alpha_1} \exp\left(- t_0 s^{1/\alpha_1}\right)\frac{\vert v(s)\vert}{s^2+\xi^2}\, ds \\
 	&\ge  \int_{s_0}^\infty  s^{(1-\beta)/\alpha_1} \exp\left(- t s^{1/\alpha_1}\right)\frac{\vert v(s)\vert}{s^2+\xi^2}\, ds.
\end{align*}
Now move the  integral $ \int_{s_0}^\infty\dots$  to the left-hand side; this gives $I(t)>0$. Hence $\tilde E(t)>0$ for $t>t_0$ and we are done.
\end{proof}

We can now prove our variant of Lemma~\ref{lem:v0w0}.

\begin{lemma}\label{lem:IVPhomog}
Consider the homogeneous version of  the initial-value problem \eqref{IVP}:
\beq\label{IVP2}
D_t^{\bar\alpha}  y(t) + \lambda y(t) = 0\ \text{ for }0<t\le T,  \quad y(0)=1,
\eeq
where the $q_i$ are constants and $\lambda\ge 0$.
Then  this problem has a solution $y$, with $y(t)\ge 0$ for $t\in[0,T]$.
\end{lemma}
\begin{proof}
If $\lambda=0$ then $y(t)\equiv 1$ is the unique solution of~\eqref{IVP2} by \cite[Theorem 4.1]{LG99}. Thus we can assume that $\lambda>0$.
From \cite[Theorem 6]{Luch11} the solution of \eqref{IVP2} is
\[
y(t) = 1- \lambda t^{\alpha_1} E_{(\alpha_1,\alpha_1-\alpha_2,\alpha_1-\alpha_3,\dots,\alpha_1-\alpha_\ell),1+\alpha_1}
		(-\lambda t^{\alpha_1},-q_2t^{\alpha_1-\alpha_2},-q_3t^{\alpha_1-\alpha_3},\dots,-q_\ell t^{\alpha_1-\alpha_\ell}).
\]
But \cite[Lemma 3.1]{LLY15} states that for $m\ge 1$ one has
\[
\frac1{\Gamma(\beta_0)} + \sum_{j=1}^m z_jE_{(\beta_1, \dots,\beta_m),\beta_0+\beta_j}(z_1, \dots, z_m)
	= E_{(\beta_1, \dots,\beta_m),\beta_0}(z_1, \dots, z_m)
\]
for $0 <\beta_0 <2$ and $0<\beta_j<1$ ($j=1,\dots,m$) and any  $z_j\in\bR$. In particular this implies that
\begin{align*}
1 &- \lambda t^{\alpha_1} E_{(\alpha_1,\alpha_1-\alpha_2,\alpha_1-\alpha_3,\dots,\alpha_1-\alpha_\ell),1+\alpha_1}
		(-\lambda t^{\alpha_1},-q_2t^{\alpha_1-\alpha_2},-q_3t^{\alpha_1-\alpha_3},\dots,-q_\ell t^{\alpha_1-\alpha_\ell}) \\
	&- \sum_{j=2}^\ell q_jt^{\alpha_1-\alpha_j}
		E_{(\alpha_1,\alpha_1-\alpha_2,\alpha_1-\alpha_3,\dots,\alpha_1-\alpha_\ell),1+\alpha_1-\alpha_j}
		(-\lambda t^{\alpha_1},-q_2t^{\alpha_1-\alpha_2},-q_3t^{\alpha_1-\alpha_3},\dots,-q_\ell t^{\alpha_1-\alpha_\ell}) \\
	&\qquad= E_{(\alpha_1,\alpha_1-\alpha_2,\alpha_1-\alpha_3,\dots,\alpha_1-\alpha_\ell),1}
		(-\lambda t^{\alpha_1},-q_2t^{\alpha_1-\alpha_2},-q_3t^{\alpha_1-\alpha_3},\dots,-q_\ell t^{\alpha_1-\alpha_\ell})
\end{align*}
Hence, using Remark~\ref{rem:Esymm}, we get
\begin{align*}
y(t) &= \sum_{j=2}^\ell q_jt^{\alpha_1-\alpha_j}
		E_{(\alpha_1,\alpha_1-\alpha_2,\alpha_1-\alpha_3,\dots,\alpha_1-\alpha_\ell),1+\alpha_1-\alpha_j}
		(-\lambda t^{\alpha_1},-q_2t^{\alpha_1-\alpha_2},-q_3t^{\alpha_1-\alpha_3},\dots,-q_\ell t^{\alpha_1-\alpha_\ell}) \\
	&\qquad+ E_{(\alpha_1,\alpha_1-\alpha_2,\alpha_1-\alpha_3,\dots,\alpha_1-\alpha_\ell),1}
		(-\lambda t^{\alpha_1},-q_2t^{\alpha_1-\alpha_2},-q_3t^{\alpha_1-\alpha_3},\dots,-q_\ell t^{\alpha_1-\alpha_\ell})\\
	&= \sum_{j=2}^\ell q_jt^{\alpha_1-\alpha_j}
		E_{(\alpha_1,\alpha_1-\alpha_\ell,\dots,\alpha_1-\alpha_3,\alpha_1-\alpha_2),1+\alpha_1-\alpha_j}
		(-\lambda t^{\alpha_1},-q_\ell t^{\alpha_1-\alpha_\ell},\dots,-q_3t^{\alpha_1-\alpha_3},-q_2 t^{\alpha_1-\alpha_2}) \\
	&\qquad+ E_{(\alpha_1,\alpha_1-\alpha_\ell,\dots,\alpha_1-\alpha_3,\alpha_1-\alpha_2),1}
		(-\lambda t^{\alpha_1},-q_\ell t^{\alpha_1-\alpha_\ell},\dots,-q_3t^{\alpha_1-\alpha_3},-q_2 t^{\alpha_1-\alpha_2})\\
	&= \sum_{j=2}^\ell  q_j \FF_{(\alpha_1,\alpha_1-\alpha_\ell, \dots,\alpha_1-\alpha_3,\alpha_1-\alpha_2),1+\alpha_1-\alpha_j}
			(t; \lambda, q_\ell, \dots, q_3, q_2) \\
	&\qquad + \FF_{(\alpha_1,\alpha_1-\alpha_\ell, \dots,\alpha_1-\alpha_3,\alpha_1-\alpha_2),1}
			(t; \lambda, q_\ell, \dots, q_3, q_2).
\end{align*}
The result now follows by applying Lemma~\ref{lem:beta<1} to the term $\FF_{(\dots),1}$ and Lemma~\ref{lem:beta>1} to each term $\FF_{(\dots),1+\alpha_1-\alpha_j}$.
\end{proof}

\bibliographystyle{plain}
\bibliography{multitermApost}

\section*{Declarations}
\begin{itemize}
\item Conflict of interest -- The authors declare that they have no conflict of interest.
\item Availability of data and materials -- Not applicable.
\end{itemize}

\end{document}